\documentclass[11pt]{amsart}

\usepackage{amsmath, amssymb, amsthm, amsfonts, enumerate, color, comment}
\usepackage{mathrsfs}
\usepackage{parskip}

\usepackage{mathtools}
\mathtoolsset{showonlyrefs}

\usepackage{palatino}
\usepackage{graphicx}

\usepackage{parskip}

\numberwithin{equation}{section}
\theoremstyle{plain}
\newtheorem{Proposition}[equation]{Proposition}
\newtheorem{Corollary}[equation]{Corollary}
\newtheorem*{Corollary*}{Corollary}
\newtheorem{Theorem}[equation]{Theorem}
\newtheorem*{Theorem*}{Theorem}
\newtheorem{Lemma}[equation]{Lemma}
\theoremstyle{definition}
\newtheorem{Definition}[equation]{Definition}

\newtheorem{Example}[equation]{Example}
\newtheorem{Remark}[equation]{Remark}

\usepackage{enumitem}
\setlist[enumerate]{leftmargin=*}
\setlist[itemize]{leftmargin=*}

\setlist[enumerate,1]{label=(\alph*),font=\upshape}

\setlist[enumerate,2]{label=(\roman*),font=\upshape}

\def\C{\mathbb{C}}
\def\R{\mathbb{R}}

\def\T{\mathbb{T}}
\def\N{\mathbb{N}}
\def\Z{\mathbb{Z}}

\def\K{\mathcal{K}}

\def\m{\mathcal{M}}

\newcommand{\h}{\mathcal{H}}

\newcommand{\M}{\emph{\textbf{M}}}
\newcommand{\LL}{\emph{\textbf{L}}}
\newcommand{\II}{\mathcal{I}}
\newcommand{\U}{\emph{\textbf{U}}}
\newcommand{\CC}{\emph{\textbf{C}}}

\newcommand{\CCC}{\mathfrak{C}}
\newcommand{\f}{\emph{\textbf{f}}}
\newcommand{\J}{\emph{\textbf{J}}}
\newcommand{\I}{\emph{\textbf{I}}}

\newcommand{\A}{\emph{\textbf{A}}}

\renewcommand{\leq}{\leqslant}
\renewcommand{\geq}{\geqslant}
\renewcommand{\subset}{\subseteq}
\renewcommand{\phi}{\varphi}
\renewcommand{\vec}[1]{{\bf #1}}

\DeclareMathOperator{\Arg}{\text{\rm Arg}}
\DeclareMathOperator{\sgn}{\text{\rm sgn}}
\usepackage{xcolor}

\author[J. Mashreghi]{Javad Mashreghi}
\address{D\'epartement de math\'ematiques et de statistique, Universit\'e Laval, Qu\'ebec, QC,
Canada, G1K 0A6}
\email{javad.mashreghi@mat.ulaval.ca}

\author[M. Ptak]{Marek Ptak}
\address{Department of Applied Mathematics,
University of Agriculture, ul. Balicka 253c\\ 30-198 Krak\'ow, Poland.}
\email{rmptak@cyf-kr.edu.pl}

\author[W. Ross]{William T. Ross}
	\address{Department of Mathematics and Computer Science, University of Richmond, Richmond, VA 23173, USA}
	\email{wross@richmond.edu}
	

\title{Conjugations of unitary operators, II}


\thanks{This work was supported by the NSERC Discovery Grant (Canada) and by the Ministry of Science and
Higher Education of the Republic of Poland.}

\begin{document}

\begin{abstract}
For a given unitary operator $U$ on a separable complex Hilbert space $\h$, we describe the set $\mathscr{C}_{c}(U)$  of all conjugations $C$ (antilinear, isometric, and involutive maps) on $\h$ for which $C U C = U$. As this set might be empty, we also show that $\mathscr{C}_{c}(U) \not = \varnothing$ if and only if $U$ is unitarily equivalent to $U^{*}$.
\end{abstract}

\maketitle

\maketitle

\section{Introduction}

This is the second in a series of two papers that  explore conjugations of unitary operators on separable complex Hilbert spaces. The first paper \cite{MPRCOUI} in this series  explored, for a given unitary operator $U$ on a Hilbert space $\h$, the antilinear, isometric, and involutive maps $C$ on $\h$, i.e.,  {\em conjugations}, for which $C U C = U^{*}$. An argument with the spectral theorem says there will always be a conjugation $C$ with this property. Moreover, \cite{MPRCOUI} contains several characterizations of the set of {\em all} such conjugations $C$ for which $C U C = U^{*}$. These conjugations are known as the ``symmetric conjugations'' for $U$.

The purpose of this second paper is to explore, for a given unitary operator $U$ on $\h$, the set
\begin{equation}\label{Comspodifposerccc}
\mathscr{C}_c(U) := \{\mbox{$C$ is a conjugation on $\mathcal{H}$}: C U C = U\}.
\end{equation}
These are known as the ``commuting conjugations'' for $U$. The subscript $c$ in the definition of $\mathscr{C}_{c}(U)$ might initially seem superfluous but we will use it anyway  to distinguish this set from $\mathscr{C}_{s}(U)$ (notice the $s$ in the subscript), the ``symmetric conjugations'' mentioned in the previous paragraph.
For an easy example of a commuting conjugation, consider the unitary operator $(U f)(\xi) = \xi f(\xi)$, the bilateral shift on $L^2(m, \T)$, where $m$ is normalized Lebesgue measure on the unit circle $\T$. One can check that the map
$$(J f)(\xi) = \overline{f(\overline{\xi})}$$ on $L^2(m, \T)$ defines a conjugation which satisfies  $J U J = U$.  Moreover (see Example \ref{sdjkfksdlfkHHGGHVGBVGHBGH}),  any conjugation $C$ on $L^2(m, \T)$ for which $C U C = U$ takes the form $(C f)(\xi) = u(\xi) (J f)(\xi)$, where $u \in L^{\infty}(m, \T)$ is both unimodular  and satisfies 
$u(\xi) = u(\overline{\xi})$ almost everywhere on $\T$. An analogous result holds when  $(U \f)(\xi) = \xi \f(\xi)$ on the vector-valued Lebesgue space $\mathscr{L}^2(m, \h)$ (Theorem \ref{jjHHbbhGHJKK}) but not always on $\mathscr{L}^2(\mu, \h)$ for a general positive measure on $\T$ (see \S \ref{section3} and the discussion below).

The  first issue one needs to resolve is whether, for a given unitary operator $U$ on $\h$, there are {\em any} conjugations $C$ for which $C U C = U$.  Indeed, using the known fact from \cite{MR0190750} (see also Proposition \ref{GL} below)  that any unitary operator can be written as a composition of two conjugations, one can fashion a quick argument (see Lemma \ref{uustar}) to see that if $\mathscr{C}_{c}(U) \not = \varnothing$, then $U \cong U^{*}$ (i.e., $U$ is unitarily equivalent to its adjoint $U^{*}$). One of the main results of this paper (Corollary \ref{uustar1})  is the converse.

\begin{Theorem}\label{fdsgdfuuuuU}
For a unitary operator $U$ on a complex separable Hilbert space $\h$, the following are equivalent.
\begin{enumerate}
\item $\mathscr{C}_{c}(U) \not = \varnothing$;
\item $U \cong U^{*}$.
\end{enumerate}
\end{Theorem}
Notice how condition (b) in Theorem \ref{fdsgdfuuuuU} places some restrictions on the class of unitary operators which have commuting conjugations in that, at the very least, the spectrum $\sigma(U)$ of $U$ must be symmetric with respect to the real axis, i.e., $\lambda \in \sigma(U)$ if and only if $\overline{\lambda} \in \sigma(U)$ (since if $U \cong U^{*}$ then $\sigma(U) = \sigma(U^{*}) = \overline{\sigma(U)}$). Thus, as an example, for the bilateral shift $(U f)(\xi) = \xi f(\xi)$ on $L^2(\mu, \T)$,  where $\mu$ is a finite positive Borel measure on $\T$, a standard argument shows that
$\sigma(U) = \{\xi \in \T: \mu(I_{\delta}(\xi)) > 0$ for all $\delta > 0\}$ ($I_{\delta}(\xi)$ is the arc of the circle centered at $\xi$ of radius $\delta$). Thus, for example, if the measure $\mu$ is supported on the top half of $\T$, then there are {\em no} conjugations $C$ on $L^2(\mu, \T)$ for which $C U C = U$. Of course, one could also consider the easy example of a unitary matrix whose eigenvalues are not symmetric with respect to the real axis. 

For a unitary operator $U$ on a finite dimensional Hilbert space, where we can regard, via a matrix representation with respect to an orthonormal basis, $U$ as a unitary matrix,  we can use the linear algebra version of the spectral theorem to see that $U \cong U^{*}$ if and only if 
$$U = W 
\begin{bmatrix}
\begin{bmatrix}
\xi_1 I_{n_1} & \\
 & \overline{\xi_1} I_{n_1}
\end{bmatrix} & & \\
& 
 \ddots & & \\
& & & \begin{bmatrix}
\xi_d I_{n_d} & \\
 & \overline{\xi_d} I_{n_d}
\end{bmatrix}\\
& & & &
\begin{bmatrix}
I_{\ell} & \\
 & - I_{k}
\end{bmatrix}
\end{bmatrix}
W^{*},$$
where $W$ is a unitary matrix, $\xi_1, \ldots, \xi_d \in \T \setminus \{1, -1\}$ are distinct eigenvalues of $U$, $I_{m}$ detnoes the $m \times m$ identity matrix, and the block in the lower right corner might not appear, or might appear as just $I_{\ell}$ or just $-I_{k}$, depending on whether $1$ or $-1$ are eigenvalues of $U$. Of course $n_j$, $\ell$, and $k$ represent the  multiplicities of their respective eigenvalues.  As we will prove in Theorem \ref{osidfgisdufbbbBB}, such unitary matrices satisfy $\mathscr{C}_{c}(U) \not = \varnothing$ and every $C \in \mathscr{C}_{c}(U)$ takes the form 
$$C = W
\begin{bmatrix}
\begin{bmatrix}
&  V_{1} \\
 V_{1}^{t}  &
\end{bmatrix} & & \\
& 
 \ddots & & \\
& & & \begin{bmatrix}
&  V_{d} \\
V_{d}^{t} &  
\end{bmatrix}\\
& & & &
\begin{bmatrix}
Q_{\ell} & \\
 & Q_k
\end{bmatrix}
\end{bmatrix}
J W^{*},$$
where each $V_j$ is an $n_j \times n_j$ unitary matrix, $Q_{\ell}$, $Q_{k}$ are $\ell \times \ell$ and $k \times k$ (respectively) unitary matrices with $Q_{\ell}^{t} = Q_{\ell}$ and $Q_{k}^{t} = Q_{k}$ (in which only one or perhaps both might not appear depending whether $1$ or $-1$ are eigenvalues of $U$), and $J$ is the conjugation on $\C^n$ defined by $J \vec{x} = \overline{\vec{x}}$ (complex conjugating each of the entries of $\vec{x}$). 

Another basic type of unitary operator is $(U \f)(\xi) = \xi \f(\xi)$ on $\mathscr{L}^2(\mu, \h)$, where $\mu$ is a positive finite Borel measure on $\T$ and $\h$ is a Hilbert space (see \S \ref{section3} for the precise definitions). As discussed earlier, there might not be any commuting conjugations for $U$. In \S \ref{section3} we discuss the restrictions one must place on $\mu$ so that  $\mathscr{C}_{c}(U) \not = \varnothing$ and, when these conditions are satisfied,  describe $\mathscr{C}_{c}(U)$. Since the operators $\f(\xi) \mapsto \xi \f(\xi)$ on  these $\mathscr{L}^2(\mu, \h)$ spaces are the building blocks for any unitary operator on a general Hilbert space, via the spectral theorem, we describe the commuting conjugations (when they exist) for a general unitary operator in \S \ref{spectraltheoremre} (where we also prove Theorem \ref{fdsgdfuuuuU}).

A particularly interesting class of  unitary operators are the multiplication operators $M_{\psi} f= \psi f$ on $L^2(m, \T)$ where $\psi$ is an inner function. Here  one has the added connection to the theory of model spaces $H^2 \cap (\psi H^2)^{\perp}$ \cite{MR3526203}. As discussed in \cite{JMMPWR_OT28}, these multiplication operators serve as models for general bilateral shifts. In \S \ref{innerfunctions} we show that $M_{\psi}$ is unitarily equivalent to its adjoint (and hence $\mathscr{C}_{c}(M_{\psi}) \not = \varnothing$ via Theorem \ref{fdsgdfuuuuU}) and proceed to  give a concrete description of  $\mathscr{C}_{c}(M_{\psi})$ (Theorem \ref{089foidjgigifghjhjh77766}). 

In the last section of this paper, we work out a concrete description of $\mathscr{C}_{c}(\mathcal{F})$ for the classical Fourier--Plancherel transform $\mathscr{F}$ on $L^2(\R)$ (Example \ref{Foureoirtert1}) and a description of $\mathscr{C}_{c}(\mathscr{H})$ for the classical Hilbert transform $\mathscr{H}$  on $L^2(\R)$ (Example \ref{Hilertrt2}).

\section{Basics facts about conjugations}\label{basicfactsdsd}

{\em All Hilbert spaces $\h$ in this paper are separable and complex.} Let $\mathcal{B}(\h)$ denote the space of all bounded linear transformations on $\h$ and $\mathscr{A}\!\mathcal{B}(\mathcal{H})$ denote  the space of all {\em bounded antilinear transformations} on $\mathcal{H}$. By this we mean that
$C \in \mathscr{A}\!\mathcal{B}(\mathcal{H})$ when
$C(\vec{x} + \alpha \vec{y}) = C \vec{x} + \overline{\alpha} C \vec{y}$ for all $\vec{x}, \vec{y} \in \mathcal{H}$  and  $\alpha \in \C$ ($C$ is antilinear)
and  $\sup\{\|C \vec{x}\|: \|\vec{x}\| = 1\}$ is finite ($C$ is bounded).
We say that $C \in \mathscr{A}\!\mathcal{B}(\mathcal{H})$ is a {\em conjugation} if it satisfies the additional conditions that $\|C \vec{x}\| = \|\vec{x}\|$ for all $\vec{x} \in \mathcal{H}$ ($C$ is isometric) and $C^2 = I$ ($C$ is involutive).
By the polarization identity, a conjugation also satisfies
\begin{equation}\label{CCCCCCcccc}
\langle C \vec{x}, C \vec{y}\rangle = \langle \vec{y}, \vec{x}\rangle \; \mbox{for all $\vec{x}, \vec{y} \in \mathcal{H}.$}
\end{equation}
Conjugations play an important role in operator theory and were initially studied in  \cite{MR2186351, MR2198373,  MR3254868,MR2187654, MR2302518}. More recently, conjugations were explored in \cite{MR4083641, MR4169409, MR4493877, MR4502709, MR4133627, MR4077555}.

\begin{Example}\label{exkjsdf77y}
Many types of conjugations  were outlined in \cite{MR3254868, MR2187654, MR2302518}. Below are a few basic ones that are relevant to this paper.
\begin{enumerate}
\item The mapping $C f = \bar{f}$ defines a conjugation on a standard Lebesgue space $L^2(\mu, X)$. In particular, the mapping 
$$C [x_1 \; x_2 \; \cdots \; x_n]^{t} = [\overline{x_1} \; \overline{x_2} \; \cdots \; \overline{x_n}]^t$$ defines a conjugation on $\C^n$. Throughout this paper we will use the symbol $t$ to represent the transpose of a matrix. In addition,  vectors in $\C^n$ will be viewed  as column vectors since, for an $n \times n$ matrix $A$ of complex numbers, we will often consider linear transformations on $\C^n$ defined by $\vec{x} \mapsto A \vec{x}$.
\item The mapping $(C f)(\xi) = \overline{f(\overline{\xi})}$ defines a conjugation on $L^2(\mu, \T)$ for any finite positive Borel measure on $\T$. 
\item On $L^2(\R)$ one can consider the two conjugations
$(C f)(t) = \overline{f(t)}$ and $(C f)(t) = \overline{f(-t)}$.
These were used in  \cite{Bender_2007, BB1998} to study symmetric operators and their connections to physics.
\end{enumerate}
\end{Example}

This next lemma  enables us to transfer a conjugation on one Hilbert space to a conjugation on another. The (easy) proof is left to the reader.

\begin{Lemma}\label{lem1.2}
Suppose $\mathcal{H}$ and $\mathcal{K}$ are Hilbert spaces and $V: \mathcal{H} \to \mathcal{K}$ is a unitary operator. If $C$ is a conjugation on $\mathcal{H}$ then $VCV^{*}$ is a conjugation on $\mathcal{K}$.
\end{Lemma}

\begin{Example}\label{JHKSJDF99}
We have already discussed the how the mapping $(C f)(\xi) = \overline{f(\overline{\xi})}$ on $L^2(m, \T)$ is a conjugation that commutes with the bilateral shift $(U f)(\xi) = \xi f(\xi)$. Here are a few other examples. 
\begin{enumerate}
\item The conjugation $(C f)(x) = \overline{f(x)}$ on $L^2(\R)$ commutes with the unitary operator $(U f)(x) = f(x - 1)$. This conjugation also commutes with the Hilbert transform. 
\item The conjugation $(C f)(x) = \overline{f(-x)}$ on $L^2(\R)$ commutes with the Fourier--Plancherel transform.
\end{enumerate}
\end{Example}

Recalling the definition of $\mathscr{C}_{c}(U)$ from \eqref{Comspodifposerccc}, let us make a few elementary observations. One can argue from  \eqref{CCCCCCcccc} that 
\begin{equation}\label{bsdv777ds77d7d7d73}
\mathscr{C}_{c}(U) = \mathscr{C}_{c}(U^{*}).
\end{equation}
Next we comment that the commuting conjugations are stable under unitary equivalence. 

\begin{Proposition}\label{bnjiovbhucfgyu}
Suppose $U, V, W$ are unitary operators on $\h$ such that $W U W^{*} = V$. Then $W \mathscr{C}_{c}(U) W^{*} = \mathscr{C}_{c}(V)$.
\end{Proposition}

If $U$ is unitary and $C$ is a conjugation on $\h$, then $UC \in \mathscr{A}\!\mathcal{B}(\mathcal{H})$ and is isometric. This next result has a straightforward proof and determines when $UC$ is involutive and hence a conjugation.

\begin{Lemma}\label{csu}
  Let $U$ be a unitary operator and $C$ be a conjugation on $\mathcal{H}$. Then $UC$ is a conjugation if and only if $C U C = U^{*}$.
\end{Lemma}

We recall the following result from \cite{MR0190750} (also see the proof of Proposition 2.5 from \cite{MPRCOUI}) which shows that any unitary operator can be built from conjugations.

\begin{Proposition}\label{GL}Let $U$ be a unitary operator on $\mathcal{H}$.  Then there are conjugations $J_1$ and $J_2$ on $\mathcal{H}$ such that $U=J_1J_2.$
Moreover,  $J_1 U J_1 = U^{*}$ and $J_2 U J_2 = U^{*}$.
\end{Proposition}

In the introduction we showed that although every unitary operator $U$ satisfies $C U C = U^{*}$ with respect to some conjugation $C$, it is possible for $\mathscr{C}_{c}(U)$ (the commuting conjugations for $U$) to be the empty set. Below we begin to determine when this happens (and bring this discussion to fruition in Corollary \ref{uustar1}).

\begin{Lemma}\label{uustar}
If $U$ is a unitary operator on $\h$ and $\mathscr{C}_{c}(U) \not = \varnothing$,  then $U \cong U^{*}$.
 \end{Lemma}
\begin{proof}Let $J_1$ be as in Proposition \ref{GL}, $C \in \mathscr{C}_{c}(U)$,  and define $V=J_1C$. Clearly $V$ is unitary (since it is linear, isometric, and onto) and
$VU =J_1CU=J_1UC=U^{*}J_1C=U^*V$. Thus, $U \cong U^{*}$. 
\end{proof}

\section{Commuting conjugations of unitary matrices}

For an $n \times n$ unitary matrix $U$, the condition as to when $\mathscr{C}_{c}(U)$ is nonempty, along with the description of $\mathscr{C}_{c}(U)$, is straightforward and so we work it out in this separate section. We begin with the following result from  \cite[Lemma 3.2]{MR2966041}.

\begin{Proposition}\label{KHADFKDJSHF}
A mapping $C$ on $\C^n$  is a conjugation if and only if $C = V J$, where $V$ is an $n \times n$ unitary matrix with $V^{t} = V$ and $J$ is the conjugation on $\C^n$ defined by 
\begin{equation}\label{JjJJJJJJjjj}
J[x_1\; x_2\; \cdots\; x_n]^{t} = [\overline{x_1}\; \overline{x_2}\; \cdots\; \overline{x_n}]^{t},
\end{equation}
i.e., 
$C[x_1 \; x_2 \; \cdots \; x_n]^{t} = V [\overline{x_1} \; \overline{x_2} \; \cdots \; \overline{x_n}]^{t}.$
\end{Proposition}

We now establish when $\mathscr{C}_{c}(U) \not = \varnothing$ for an $n \times n$ unitary matrix $U$. This is a special case of Theorem  \ref{fdsgdfuuuuU}.

\begin{Proposition}\label{propositionCunonsmemepry}
For an $n \times n$ unitary matrix $U$ the following are equivalent. 
\begin{enumerate}
\item $\mathscr{C}_{c}(U) \not = \varnothing$; 
\item $U \cong U^{*}$. 
\end{enumerate}
\end{Proposition}

\begin{proof}
The implication $(a) \Longrightarrow (b)$ is from Lemma \ref{uustar}. For the implication $(b) \Longrightarrow (a)$,  suppose that $U \cong U^{*}$. As mentioned in the introduction, the spectral theorem for unitary matrices implies that $U$ is unitarily equivalent to 
\begin{equation}\label{Uprimepwerteww}
U' = 
\begin{bmatrix}
\begin{bmatrix}
\xi_1 I_{n_1} & \\
 & \overline{\xi_1} I_{n_1}
\end{bmatrix} & & \\
& 
 \ddots & & \\
& & & \begin{bmatrix}
\xi_d I_{n_d} & \\
 & \overline{\xi_d} I_{n_d}
\end{bmatrix}\\
& & & &
\begin{bmatrix}
I_{\ell} & \\
 & - I_{k}
\end{bmatrix}
\end{bmatrix},
\end{equation}
where $\xi_1, \ldots, \xi_d \in \T \setminus \{1, -1\}$ are distinct eigenvalues of $U$, $I_{m}$ denotes the $m \times m$ identity matrix, and the block in the lower right corner might not appear or might appear as just $I_{\ell}$ or just $-I_{k}$, depending on whether $1$ or $-1$ are eigenvalues of $U$. Of course $n_j$, $\ell$, and $k$ represent the  multiplicities of the respective eigenvalues and $2 n_1 + \cdots + 2 n_d + \ell + k = n$.

Now consider the mapping 
\begin{equation}\label{00088ooLLLLL}
C' = \begin{bmatrix}
\begin{bmatrix}
&  I_{n_1} \\
I_{n_1}  &
\end{bmatrix} & & \\
& 
 \ddots & & \\
& & & \begin{bmatrix}
&  I_{n_d} \\
I_{n_d} &  
\end{bmatrix}\\
& & & &
\begin{bmatrix}
I_{\ell} & \\
 & I_{k}
\end{bmatrix}
\end{bmatrix}
J,
\end{equation}
where $J$ is the conjugation on $\C^n$ from \eqref{JjJJJJJJjjj}. Proposition \ref{KHADFKDJSHF} says that $C'$ is a conjugation on $\C^n$ and block multiplication will show that $C' U' C' = U'$. If $W$ is the unitary matrix such that $W U' W^{*} = U$,  Lemma \ref{lem1.2} says  that $C = W C' W^{*}$ is a conjugation on $\C^n$ and Proposition \ref{bnjiovbhucfgyu} shows that $C U C = U$. Thus, $\mathscr{C}_{c}(U) \not = \varnothing$.
\end{proof}

\begin{Theorem}\label{osidfgisdufbbbBB}
Suppose that $U$ is an $n \times n$ unitary matrix with $U  \cong U^{*}$ and $W$ is a unitary matrix such that $W U W^{*} = U'$, where $U'$ is the matrix from \eqref{Uprimepwerteww}. Then every $C \in \mathscr{C}_{c}(U)$ takes the form 
\begin{equation}\label{7377GGGGG777g}
C = W
\begin{bmatrix}
\begin{bmatrix}
&  V_{1} \\
 V_{1}^{t}  &
\end{bmatrix} & & \\
& 
 \ddots & & \\
& & & \begin{bmatrix}
&  V_{d} \\
V_{d}^{t} &  
\end{bmatrix}\\
& & & &
\begin{bmatrix}
Q_{\ell} & \\
 & Q_k
\end{bmatrix}
\end{bmatrix}
J W^{*},
\end{equation}
where each $V_j$ is an $n_j \times n_j$ unitary matrix, $Q_{\ell}$, $Q_{k}$ are $\ell \times \ell$ and $k \times k$ (respectively) unitary matrices with $Q_{\ell}^{t} = Q_{\ell}$ and $Q_{k}^{t} = Q_{k}$ (in which only one or perhaps both might not appear depending whether $1$ or $-1$ are eigenvalues of $U$), and $J$ is the conjugation on $\C^n$ from \eqref{JjJJJJJJjjj}.
\end{Theorem}

\begin{proof}
By Lemma \ref{lem1.2} and Proposition \ref{JjJJJJJJjjj}, the mapping $C$ from \eqref{7377GGGGG777g} defines a conjugation on $\C^n$. As argued in the proof of Proposition \ref{propositionCunonsmemepry} (replacing the conjugation $C$ from \eqref{7377GGGGG777g} with the conjugation $C'$ from \eqref{00088ooLLLLL}), one can see that each conjugation $C$ from \eqref{7377GGGGG777g} belongs to $\mathscr{C}_{c}(U)$. 

Now suppose that $C \in \mathscr{C}_{c}(U)$. Then $C' := W^{*} C W \in \mathscr{C}_{c}(U')$, where $U'$ is the matrix from \eqref{Uprimepwerteww}. From Proposition \ref{KHADFKDJSHF}, $C' = V J$, where $V$ is an $n \times n$ unitary matrix with $V^{t} = V$. Now observe that $J U' J = \overline{U'}$ (the matrix $U'$ with all the entries conjugated) and $J V J = \overline{V} = \overline{V^t} = V^{*}$ and thus 
$$U' = C' U' C' = (V J) U' (V J )= V (J U' J) (J V J) = V \overline{U'} V^{*}.$$ This yields the identity $U' V = V \overline{U'}$. A computation with block multiplication of matrices and the fact that $V = V^{t}$ (along with the facts that $U'$ is block diagonal with distinct multiples of the identity matrix along its diagonal) will show that 
$$V = \begin{bmatrix}
\begin{bmatrix}
&  V_{1} \\
 V_{1}^{t}  &
\end{bmatrix} & & \\
& 
 \ddots & & \\
& & & \begin{bmatrix}
&  V_{d} \\
V_{d}^{t} &  
\end{bmatrix}\\
& & & &
\begin{bmatrix}
Q_{\ell} & \\
 & Q_k
\end{bmatrix}
\end{bmatrix},$$
where $V_{j}$ are $n_j \times n_j$ unitary matrices and $Q_{\ell}$ and $Q_{k}$ are $\ell \times \ell$ and $k \times k$ unitary matrices with $Q_{\ell}^{t} = Q_{\ell}$, $Q_{k}^{t} = Q_{k}$.
Thus, $C = W C' W^{*}$ has the desired form in \eqref{7377GGGGG777g}.
\end{proof}

\section{Conjugations and spectral measures}

A version of the spectral theorem for unitary operators  \cite[Ch.~IX, Thm.~2.2]{ConwayFA}  (see also \cite{MR0045309}) says that if $U$ is a unitary operator on $\mathcal{H}$, then there is a unique spectral measure $E(\cdot)$ on $\T$ such that
\begin{equation}\label{9sectraloth}
U = \int \xi d E(\xi).
\end{equation}
Moreover, for any spectral measure $E(\cdot)$ on $\T$ (i.e., $E(\cdot)$ is a projection-valued function on the Borel subsets of $\T$ such that $E(\T) = I$ and $E(\cdot)$ is countably additive), there is a unique unitary operator $U$ associated  with $E(\cdot)$ via \eqref{9sectraloth}.

 For a spectral measure $E(\cdot)$ and $\vec{x}, \vec{y} \in \mathcal{H}$, the function
$$\mu_{\vec{x}, \vec{y}}(\cdot) := \langle E(\cdot) \vec{x}, \vec{y}\rangle$$
 defines a finite complex Borel measure on $\T$ and, in particular, for each  $\vec{x} \in \h$,
\begin{equation}\label{67uyh87877YUYUY}
\mu_{\vec{x}} :=\mu_{\vec{x},\vec{x}}
\end{equation} defines a  finite positive Borel  measure on $\T$, sometimes called an {\it elementary measure}.

For a complex-valued Borel measure $\mu$ on $\T$, define a new complex Borel measure $\mu^{c}$ on the Borel subsets $\Omega$ of  $\T$ by
 \begin{equation}\label{muc}
 \mu^{c}(\Omega) := \mu(\Omega^{*}), \; \; \mbox{where} \; \; 
 \Omega^{*} : = \{\overline{\xi}: \xi \in \Omega\}
 \end{equation}
One can argue that $(\mu^{c})^{c} = \mu$.  For a spectral measure $E(\cdot)$ on $\T$, we have the  family of measures
 $\{\mu_{\vec{x}, \vec{y}}^{c}: \vec{x}, \vec{y} \in \mathcal{H}\}$ defined via \eqref{muc}.

  For the rest of this paper, we use $M_{+}(\T)$ to denote the set of all finite positive Borel measures on $\T$.

 \begin{Proposition}\label{p3.1a}
 Suppose $\mu \in M_{+}(\T)$ and  $\mu^c \ll \mu$. Then the following hold.
 \begin{enumerate}
 \item  $\mu \ll \mu^{c}$;
 \item The Radon--Nikodym derivatives satisfy 
\begin{equation}\label{rnd}
  \frac{d\mu^c}{d\mu}(\xi) \cdot \frac{d\mu^c}{d\mu}(\bar \xi)=1 \; \; \mbox{for $\mu$-almost every $\xi \in \T$.}
\end{equation}
\end{enumerate}
\end{Proposition}
\begin{proof}
Let $h = d \mu^{c}/d \mu$.
Observe that $\mu=(\mu^c)^c\ll\mu^c$ and
                   \[ d \mu(\xi)=  d \mu^c(\bar\xi)  = h(\bar \xi)d \mu(\bar\xi) = h(\bar \xi)d \mu^{c}(\xi) = h(\bar \xi) h(\xi)d \mu(\xi)\]
for $\mu$-almost every $\xi \in \T$.
                   Therefore,
 \begin{equation} \label{e0r98utiofpwkdsdgf}
 d \mu(\xi)  = h(\bar \xi)d \mu^{c}(\xi) \; \; \mbox{and} \; \;  h(\bar \xi) h(\xi)=1 
 \end{equation}
  for $\mu$-almost every $\xi \in \T$.
   \end{proof}

The following proposition, originally explored in \cite{MR0190750} for symmetric conjugations, relates  a $C \in \mathscr{C}_{c}(U)$ with the associated spectral measure $E(\cdot)$ for $U$. Define $E^{c}(\cdot)$ on Borel subsets $\Omega$ of $\T$ by 
$$E^{c}(\Omega) := E(\Omega^{*}).$$ From this definition it follows that 
$$\langle E^{c}(\Omega) \vec{x}, \vec{y}\rangle=\mu_{\vec{x}, \vec{y}}^{c}(\Omega) \; \mbox{for all $\vec{x}, \vec{y} \in \h$.}$$

\begin{Proposition}\label{4}
Let $C$ be a conjugation on $\h$ and $U$ be a unitary operator on $\h$ with associated spectral measure $E(\cdot)$. 
Then we have the following. 
\begin{enumerate}
\item 
 $E^c(\cdot)$ is the associated spectral measure for $U^*$.
\item  $CE(\cdot)C$ is the spectral measure for $CU^*C$.
\item $CUC = U^{*}$ if and only if $C E(\Omega) C = E(\Omega)$ for all Borel subsets $\Omega$ of $\T$.
\item $CUC = U$  if and only if $C E(\Omega)C = E^{c}(\Omega)$ for all Borel subsets $\Omega$ of $\T$.
\end{enumerate}
\end{Proposition}
\begin{proof}
If $E(\cdot)$ is a spectral measure, one can check that $E^c(\cdot)$ and $CE(\cdot)C$ are also spectral measures. Since, for each pair $\vec{x}, \vec{y} \in \h$,
$$\langle U^* \vec{x},\vec{y}\rangle=\int\bar\xi\,d\langle E(\xi)\vec{x},\vec{y}\rangle =\int \xi \,d\langle E^c(\xi)\vec{x},\vec{y}\rangle,$$
the uniqueness of the spectral measure for a unitary operator gives (a).
In a similar way,  (b) is a consequence of the computation 
\begin{align*}
\langle CU^*C \vec{x},\vec{y}\rangle &=\langle C \vec{y},U^*C \vec{x}\rangle\\
&=\langle UC\vec{y},C\vec{x}\rangle\\
&= \int\xi\,d\langle E(\xi)C\vec{y},C\vec{x}\rangle\\
&=\int\xi\,d\langle \vec{x},CE(\xi)C\vec{y}\rangle\\
&=\int\xi\,d\langle CE(\xi)C\vec{x},\vec{y}\rangle.
 \end{align*}
 Note the use of \eqref{CCCCCCcccc} in the above calculation. To see (c), note that $CU^*C = U$  if and only if their spectral measures $CE(\cdot)C$ and $E(\cdot)$ coincide. Symmetrically in (d), $CU^*C$ equals to $U^*$  if and only if the spectral measures $CE(\cdot)C$ and $E^c(\cdot)$ coincide.
\end{proof}

As we will see in subsequent sections, the set $\mathscr{C}_c(U)$ is quite large and so an important step in understanding it is to decompose each $C \in \mathscr{C}_{c}(U)$  into more manageable pieces. This decomposition will involve various types of invariant subspaces. Recall that a (closed) subspace $\mathcal{M}$ of a Hilbert space $\mathcal{H}$ is {\em invariant} for an $A \in \mathcal{B}(\mathcal{H})$ if $A \mathcal{M} \subset \mathcal{M}$; {\em reducing} if both $A \mathcal{M} \subset \mathcal{M}$ and $A^{*} \mathcal{M} \subset \mathcal{M}$; and {\em hyperinvariant} if $T \mathcal{M} \subset \mathcal{M}$ for every $T \in \mathcal{B}(\mathcal{H})$ that commutes with $A$. We begin with a simple lemma whose proof follows from \eqref{CCCCCCcccc} and the fact that $C^2 = I$. 

\begin{Lemma}\label{nnNDF99}
If $C$ is a conjugation on $\h$ and $\mathcal{M}$ is a subspace of $\h$ such that $C \mathcal{M} \subset \mathcal{M}$, then $C \mathcal{M} = \mathcal{M}$ and $C\mathcal{M}^{\perp}=\mathcal{M}^\perp$.
\end{Lemma}

\begin{Proposition}\label{98auifdovjlscfergeAA}
Let $U\in\mathcal{B}(\h)$ be a unitary operator with associated spectral measure $E(\cdot)$ and $\Omega\subset \T$ be a Borel set.
  \begin{enumerate}
  \item
 If ${\Omega}^*=\Omega$  then for any $C \in \mathscr{C}_{c}(U)$,  we have 
 $C(E(\Omega)\h) =  E(\Omega)\h$.
\item If $\mathscr{C}_{c}(U) \not = \varnothing$ and $E(\Omega)\h$ is invariant for $C$, then $E(\Omega\setminus \Omega^*)=0$.
\end{enumerate}
 \end{Proposition}
\begin{proof}
For the proof of (a),  let $\vec{x}\in E(\Omega)\h$ and $\vec{y}\in (E(\Omega)\h)^\perp$. By Proposition \ref{4}(d) we have
  \[  \langle C\vec{x}, \vec{y} \rangle= \langle CE(\Omega)\vec{x}, \vec{y} \rangle= \langle E(\Omega^{*})C\vec{x}, \vec{y} \rangle=  \langle C\vec{x}, E(\Omega)\vec{y} \rangle = \langle C \vec{x}, \vec{0}\rangle =0\]
and thus $C\vec{x}\in E(\Omega)\h$. Now apply Lemma \ref{nnNDF99}.

 For the proof of (b) let $\vec{x}\in E(\Omega\setminus \Omega^*)\h$. From $E(\Omega)= E(\Omega^{*}) \oplus E(\Omega \setminus \Omega^{*})$, we can use Proposition \ref{4}(d) to see that 
\begin{align*}
0 &=\|E(\Omega^*)\vec{x}\|=\|CE(\Omega)C \vec{x}\|=\|E(\Omega)C\vec{x}\|=\|C\vec{x}\|=\|\vec{x}\|. \qedhere
\end{align*}
\end{proof}

For a unitary operator $U$ on $\h$ with associated spectral measure $E(\cdot)$ and the associated family of elementary measures $\mu_{\vec{x}}, \vec{x} \in \h$ from \eqref{67uyh87877YUYUY},  one can show, as was done in \cite{MPRCOUI}, that for any $\mu \in M_{+}(\T)$ the set 
\[ \h_\mu :=\{\vec{x} \in \h: \mu_{\vec{x}} \ll \mu \}
\]
is a reducing subspace of $U$. The space $\h_{\mu}$ was discussed in \cite[\S 65]{MR0045309} as part of a general discussion of the spectral multiplicity theory for unitary operators. 

\begin{Theorem}\label{conj_dec}Let $U$ be a unitary operator on $\h$, $E(\cdot)$ its associated spectral measure, $\mu \in M_{+}(\T)$, and $C \in \mathscr{C}_{c}(U)$. Then we have the following. 
\begin{enumerate}
\item $C\h_\mu= \h_{\mu^c}$ and $C\h_\mu^{\perp}= \h_{\mu^c}^\perp$, and thus
\item
$C = C_{\mu, \mu^{c}} \oplus C^{'}_{\mu, \mu^{c}}$, where $C_{\mu, \mu^{c}}=C|_{{\h_\mu}}: \h_{\mu} \to \h_{\mu^{c}}$ and $C^{'}_{\mu, \mu^{c}}=C|_{{\h_\mu^\perp}}: \h_{\mu}^{\perp} \to \h_{\mu^c}^{\perp}$ are antilinear, onto, isometries.
\end{enumerate}
\end{Theorem}
\begin{proof}
Let $\vec{x}\in\h_\mu$. 
 By Proposition \ref{4}(d), $CE(\cdot)C=E^c(\cdot)$ and thus
\begin{align*}
\langle E(\cdot) C \vec{x}, C \vec{x}\rangle  = \langle \vec{x}, C E(\cdot) C \vec{x}\rangle
 = \langle \vec{x}, E^c(\cdot) \vec{x}\rangle
 = \langle E^c(\cdot) \vec{x}, \vec{x}\rangle.
\end{align*}
Since $\langle E(\cdot)\vec{x},\vec{x}\rangle\ll\mu$, it follows that $\langle E(\cdot) C\vec{x}, C\vec{x}\rangle\ll\mu^c$ and thus $C\vec{x}\in\h_{\mu^c}$. Similarly, $C\h_{\mu^c}\subset \h_{\mu}$, thus $C\h_\mu= \h_{\mu^c}$ and $C\h_\mu^\perp= \h_{\mu^c}^\perp$ (Lemma \ref{nnNDF99}).
\end{proof}

Recall \cite[\S 48]{MR0045309} the  standard Boolean operations $\wedge$ and $\vee$ for $\mu_1, \mu_2 \in M_{+}(\T)$ defined on Borel subsets $\Omega$ of $\T$ by 
$$(\mu_1 \vee \mu_2) (\Omega) := \mu_1(\Omega) + \mu_2(\Omega);$$
$$(\mu_1 \wedge \mu_2)(\Omega) := \inf\{\mu_1(\Omega \cap A) + \mu_2(\Omega \setminus A): \mbox{$A$ is a Borel set}\}.$$ For a unitary operator $U$,  there exists a {\em scalar spectral measure} $\nu$, meaning that $\nu(\Delta) = 0$ if and only if $E(\Delta) = 0$, where $E(\cdot)$ is the spectral measure for $U$ \cite[p.~293]{ConwayFA} (also see the discussion in Theorem 3.8 in \cite{MPRCOUI}).  For  $\nu_1, \nu_2 \in M_{+}(\T)$ it was shown in  \cite[Prop.~3.10]{MPRCOUI} that $\h_{\nu_1}\subset \h_{\nu_2}$ if and only if $\nu_1\wedge \mu \ll \nu_2\wedge \mu.$

\begin{Corollary}\label{uuUUSSSS}
Let $U$ be a unitary operator on $\h$ and $\nu$ be any scalar spectral measure for $U$. Suppose that $\mu \in M_{+}(\T)$ satisfies $\mu^{c} \wedge \nu \ll \mu \wedge \nu$. If $C \in \mathscr{C}_{c}(U)$, we have the following. 
\begin{enumerate}
\item $C \h_{\mu} = \h_{\mu}$ and $C \h_{\mu}^{\perp} = \h_{\mu}^{\perp}$.
\item $C = C_{\mu} \oplus C_{\mu}^{\perp}$, where $C_{\mu} := C|_{\h_{\mu}}$ and $C_{\mu}^{\perp} = C|_{\h_{\mu}^{\perp}}$.
\item $C_{\mu} \in \mathscr{C}_{c}(U|_{\h_{\mu}})$ and $C_{\mu}^{\perp} \in \mathscr{C}_{c}(U|_{\h_{\mu}^{\perp}})$.
\end{enumerate}
\end{Corollary}

\begin{Corollary}\label{conj_dec1}Let $U$ be a unitary operator on $\h$ and  $\nu$ be any scalar spectral measure for $U$.  Fix a  $\mu \in M_+(\T)$.
If $C\h_\mu\subset\h_\mu$ for some  $C\in \mathscr{C}_c(U)$ then $\mu^c\wedge \nu\ll\mu\wedge \nu$.
\end{Corollary}
\begin{proof}
  By Theorem \ref{conj_dec} we have  $C\h_\mu=\h_{\mu^c}\subset \h_\mu$. Thus, by \cite[Prop. 3.11]{MPRCOUI}, we obtain $\mu^c\wedge \nu\ll\mu\wedge \nu$.
\end{proof}

Since a unitary operator is normal, we see that $\ker(U-\alpha I)=\ker(U^*-\bar\alpha I)$ i.e.,  $\h_{\delta_\alpha}=\h_{{\delta_{\alpha}^c}}$, where $\delta_{\xi}$ denotes an atomic measure with atom at $\xi \in \T$. This gives us the following corollary.

\begin{Corollary} Let $U$ be a unitary operator on $\h$ and $C \in \mathscr{C}_{c}(U)$. 
  Let $\alpha \in \mathbb{T}$ be an eigenvalue for $U$.
  Then $$C=C_{\delta_\alpha}\oplus C_{\delta_\alpha}^\perp,$$ where 
  $C_{\delta_{\alpha}} = C|_{\h_{\delta_{\alpha}}}$ and $C_{\delta_{\alpha}}^{\perp} = C|_{\h_{\delta_{\alpha}}^{\perp}}$ are conjugations on $\ker(U-\alpha I)$ and $\ker(U-\alpha I)^\perp$, respectively. 
\end{Corollary}

 In the above discussion, we often have the hypothesis that $\mathscr{C}_{c}(U) \not = \varnothing$. As mentioned earlier, this is not always the case (e.g., if $U$ is not unitarily equivalent to its adjoint -- Lemma \ref{uustar}). 

  \section{Natural conjugations on vector valued $L^2$ spaces}\label{section3}

This section provides a model for conjugations on vector valued Lebesgue spaces and  will be useful in our description of $\mathscr{C}_c(U)$ in Theorem \ref{th1.2a}. This notation also sets up our discussion of models for bilateral shifts in the next section.

For a Hilbert space $\h$ with norm $\|\cdot\|_{\h}$ and a $\mu \in M_{+}(\T)$, consider the set $\mathscr{L}^{0}(\mu, \h)$ of $\h$-valued $\mu$-measurable functions $\f$ on $\T$ and the set
$$\mathscr{L}^2(\mu, \h) := \Big\{\f \in \mathscr{L}^{0}(\mu, \h): \|\f\|_{L^2(\mu. \h)} :=  \Big(\int_{\T} \|\f(\xi)\|^{2}_{\h}d\mu(\xi)\Big)^{\frac{1}{2}} < \infty\Big\}.$$ This is often described using tensor notation as $L^2(\mu) \otimes \h$. 

Also consider $\mathscr{L}^{\infty}(\mu, \mathcal{B}(\h))$, the $\mu$-essentially bounded $\mathcal{B}(\h)$-valued functions $\U$ on $\T$. For $\U \in \mathscr{L}^{\infty}(\mu, \mathcal{B}(\h))$, define the multiplication operator $\M_{\U}$ on $\mathscr{L}^2(\mu, \h)$ by
\begin{equation}\label{e1.4}
  (\M_{\U}\f)(\xi)=\U(\xi)\f(\xi)
\end{equation}
for  $\f\in \mathscr{L}^2(\mu,\h)$ and $\mu$-almost every $\xi \in \T$.
Clearly $\M_{\U}\in \mathcal{B}(\mathscr{L}^2(\mu,\h))$. If we use the notation $\U^{*}(\xi) = \U(\xi)^{*}$, one can verify that 
\begin{equation}\label{MuuMuu**}
\M_{\U}^{*} = \M_{\U^{*}}.
\end{equation} We will use $L^{\infty}(\mu) := \mathscr{L}^{\infty}(\mu, \C)$ to denote the scalar valued $\mu$-essentially bounded functions on $\T$. For ease of notation, we will write $\M_{\varphi}$, when $\phi \in L^{\infty}(\mu)$, instead of the more cumbersome $\M_{\varphi \I_{\h}}$, that is,
\begin{equation}\label{e1.5}
  (\M_\varphi \f)(\xi)=(\M_{\varphi \I_\h}\f)(\xi)=\varphi(\xi)\f(\xi)
\end{equation}
for $\f \in \mathscr{L}^2(\mu, \h)$ and $\mu$-almost every $\xi \in \T$. The case when $\phi(\xi) = \xi$ will play an prominent role in this paper in which case we have the vector-valued bilateral shift $\M_{\xi}$ on $\mathscr{L}^2(\mu, \h)$. 

Recall from \S \ref{basicfactsdsd} that $\mathscr{A}\!\mathcal{B}(\h)$ denotes the space of all bounded antilinear operators on $\h$ . We define $\mathscr{L}^{\infty}(\mu, \mathscr{A}\!\mathcal{B}(\h))$ to be the space of  all $\mu$-essentially bounded and $\mathscr{A}\!\mathcal{B}(\h)$-valued Borel functions on $\T$. Similarly as above, for $\CC\in \mathscr{L}^{\infty}(\mu, \mathscr{A}\!\mathcal{B}(\h))$, define
\begin{equation*}
  (\A_\CC \f)(\xi)=\CC(\xi)\f(\xi)
\end{equation*}
for $\f\in \mathscr{L}^2(\mu,\h)$ and $\mu$-almost every $\xi \in \T$.
One can check that  $\A_\CC\in \mathscr{A}\!\mathcal{B}(\mathscr{L}^2(\mu,\h))$.

For any conjugation $J$ on $\h$, define the conjugation $\J$ on $\mathscr{L}^2(\mu,\h)$ by
\begin{equation}\label{e1.6}
  (\J\f)(\xi)=J(\f(\xi)), \; \; \f\in \mathscr{L}^2(\mu,\h).
\end{equation}
 %
Notice that 
$  \J\M_\xi\J=\M_{\bar \xi}
$  \cite{MPRCOUI}.

We now focus our attention on the scalar valued $L^2(\mu)$ space and the set $\mathscr{C}_{c}(M_{\xi})$. This next result shows that when $\mathscr{C}_{c}(M_{\xi}) \not = \varnothing$, there must be some restrictions on $\mu$.  The set $\mathscr{C}_{c}(M_{\xi})$ was explored in  \cite{MR4169409} when $\mu = m$.

\begin{Proposition}\label{p3.1}Let $\mu \in M_{+}(\T)$ and   $C$ be a conjugation on $L^2(\mu )$ such that $C M_\xi=M_\xi C$. Then $\mu^c \ll \mu$ (and hence  $\mu \ll \mu^{c}$ by Proposition \ref{p3.1a}). 
\end{Proposition}
\begin{proof}
From \eqref{bsdv777ds77d7d7d73}, the identity  $C M_ \xi C=M_\xi$ implies that $CM_{\bar \xi}C=M_{\bar \xi}$. For any trigonometric polynomial $p(\xi)$ define 
$$p^{\#}(\xi) := \overline{p(\overline{\xi})}.$$
The above (and the antilinearity of $C$) shows that
$$CM_{p}C=M_{p^{\#}}.$$  Therefore, by the weak-$*$ density of the trigonometric polynomials  in $L^{\infty}(\mu)$, we obtain
 \begin{equation}\label{e4.1}CM_{\varphi}C=M_{{\varphi}^{\#}} \;  \text{for any $\varphi \in L^\infty(\mu)$,}\end{equation} where
\begin{equation}\label{shiprtt}
{\varphi}^{\#}(\xi)=\overline{\varphi(\overline{\xi})}.
\end{equation} If $\mu^c$ were not absolutely continuous with respect to $\mu$, then there would be a Borel set $\Omega \subset \T$ such that $\mu(\Omega)\not= 0$ but  $\mu^c(\Omega)= 0$.  However,  \eqref{e4.1} leads to contradiction with $\varphi=\chi_\Omega$, since $M_{\chi_{\Omega^*}}=0$ but $CM_{\chi_\Omega}C$ is not.
\end{proof}

\begin{Remark}
One can adapt the proof of the above proposition to the vector--valued space $\mathscr{L}^2(\mu,\h)$.
\end{Remark}

 Now let us focus on the situation  when $\mu^c\ll\mu$. In this case we also have that  $\mu\ll\mu^c$ (Proposition \ref{p3.1a}). For $\f\in \mathscr{L}^2(\mu,\h)$ and $\U\in \mathscr{L}^{\infty}(\mu, \mathcal{B}(\h))$, it  makes sense to write $\f(\bar \xi)$ or $\U(\bar \xi)$ and define 
 \begin{equation}\label{11111LL11ll}
 \U^\#(\xi):=\U^*(\bar\xi)=\U(\bar\xi)^*.
 \end{equation}
 
\begin{Proposition}\label{p3.2a}
Let  $\mu \in M_{+}(\T)$  such that  $\mu^c \ll \mu$ and let $h = d \mu^{c}/d \mu$. 
For a  Hilbert space $\h$, a conjugation $J$ on $\h$, and $\f \in \mathscr{L}^2(\mu, \h)$, define 
\begin{equation}\label{e1.7}
  (\J^{\#}\f)(\xi)=  (h(\xi))^{\frac{1}{2}}J(\f(\bar \xi))
\end{equation}
  for $\mu$-almost every $\xi \in \T$. 
 Then we have the following. 
\begin{enumerate}\item $\J^{\#}$ is a conjugation on $\mathscr{L}^2(\mu,\h)$;
\item $  \J^{\#}\M_\xi\J^{\#}=\M_\xi$.
\end{enumerate}
\end{Proposition}

\begin{proof}
As discussed in Proposition \ref{p3.1a}, $\mu\ll\mu^c$ and  $d\mu^c=h^{\#}\,d\mu$ with $h^\#(\xi)\, h(\xi)=h(\bar \xi)h(\xi)=1$ for $\mu$ almost every $\xi \in \T$.

Since $J$ is antilinear on $\h$, one sees that $\J^{\#}$ is antilinear on $\mathscr{L}^2(\mu, \h)$. Moreover, for $\f\in \mathscr{L}^2(\mu,\h)$ we have
\begin{align*} \|\J^{\#}\f\|_{\mathscr{L}^2(\mu, \h)}^{2}&=\int \|h(\xi)^{\frac 12}\,J(\f(\bar \xi))\|_{\h}^2d\mu(\xi)\\
&=  \int \|J(\f(\bar \xi))\|_{\h}^2\, h(\xi)\,d\mu(\xi)\\
&= \int \|\f( \xi)\|_{\h}^2\, h(\bar\xi)\,d\mu(\bar\xi)\\
&= \int \|\f(\xi)\|_{\h}^2\,d\mu( \xi)\\
&=\|\f\|_{\mathscr{L}^2(\mu, \h)}^2.
\end{align*}
Note the use of \eqref{e0r98utiofpwkdsdgf} above.
Thus, $\J^{\#}$ is isometric on $\mathscr{L}^2(\mu, \h)$.

Next we show that $(\J^{\#})^2=I$. Indeed, for each $\f \in \mathscr{L}^2(\mu, \h)$,
\begin{align*} (\J^{\#} \J^{\#}\f) (\xi)&=h(\xi)^{\frac 12}\, J((\J^{\#}\f)(\bar \xi))\\
&=
  h( \xi)^{\frac 12}\, J\Big((h(\bar \xi))^{\frac 12}\, J(\f( \xi))\Big)  \\
&=  (h( \xi) h(\bar \xi)^{\frac 12}\, J( J(\f( \xi)))\\
&=\f(\xi).
\end{align*}
Again, note the use of \eqref{e0r98utiofpwkdsdgf} above. 
Therefore, $\J^{\#}$ is a conjugation. To prove (b), observe that for each $\f \in \mathscr{L}^2(\mu, \h)$ we have 
\begin{align*}
(\J^{\#}\M_\xi \f)(\xi)&= \J^{\#}(\M_\xi\f)(\xi)\\
&= h( \xi)^{\frac 12}\, J((\M_\xi\f)(\bar \xi))\\
&=h( \xi)^{\frac 12}\, J(\bar \xi\f(\bar \xi))\\
& = \xi h( \xi)^{\frac 12}\, J(\f(\bar \xi))\\
&=   \xi(\J^{\#}\f)(\xi)\\
& = (\M_\xi\J^{\#})\f(\xi). \qedhere
\end{align*}

\end{proof}

\begin{Remark} If $\mu = m$, Lebesgue measure on $\mathbb{T}$, then  $m=m^c$ and $h\equiv 1$ and  the conjugation \eqref{e1.7} coincides with the one considered in \cite{MR4169409}.
\end{Remark}

A special case worth pointing out is the scalar case $\h=\mathbb{C}$.
\begin{Corollary}
Let  $\mu \in M_{+}(\T)$  such that  $\mu^c\ll\mu$.
Let $h = d\mu^{c}/d \mu$ and define
\begin{equation}\label{e1.7a}
  (J^{\#} f)(\xi)=  h( \xi)^{\frac 12}\ \overline{f(\bar \xi)}, \quad f \in L^2(\mu).
\end{equation} Then $J^{\#}$ is a conjugation on $L^2(\mu)$ and  $  J^{\#} M_\xi J^{\#}= M_\xi$.
\end{Corollary}

In particular, observe that 
$$\mu^{c} \ll \mu \Longrightarrow \mathscr{C}_{c}(M_{\xi}) \not = \varnothing.$$

The following echos a result from \cite[Proposition 4.2]{MR4169409}. Recall the notation from \eqref{11111LL11ll}.

\begin{Proposition}\label{p3.4} Let $J$ be a conjugation on $\h$, $\J^\#$ be defined by \eqref{e1.7},  and let $\U\in \mathscr{L}^{\infty}(\mu, \mathcal{B}(\h))$ be a unitary operator valued function. Then we have the following. 
\begin{enumerate}
\item $\J^\#\M_\U\J^{\#}=\J\M_{(\U^\#)^*} \J$;
  \item $\M_{\U}\J^{\#}$ is a conjugation on $\mathscr{L}^2(\mu, \h)$ if and only if 
  $$J\U( \xi)J=\U^\#( \xi)=\U^*(\bar \xi)$$ for $\mu$-almost every $\xi \in \T$;
      \item If $\M_{\U}\J^{\#}$ is a conjugation on $\mathscr{L}^2(\mu, \h)$ then  $\M_{\U}\J^{\#}=\J^{\#}\M_{\U^*}$;
  \item $\big(\M_\U\J^{\#}\big)\M_\xi\big(\M_\U\J^{\#}\big)=\M_{\xi}$.
\end{enumerate}
\end{Proposition}

\begin{proof} For every $\f \in \mathscr{L}^2(\mu, \h)$, observe that for $\mu$ almost every $\xi \in \T$ we have 
\begin{align*}
 (\J^\#\M_\U\J^{\#}  \f)(\xi)& =h( \xi)^{\frac 12}J\big((\M_\U\J^{\#} \f)(\bar \xi)\big)\\
 &=h( \xi)^{\frac 12}\,J\big(\U(\bar\xi)\big(\J^{\#} \f\big)(\bar\xi)\big)\\
 &=h( \xi)^{\frac 12}\,J\big(\U(\bar\xi)h(\bar \xi)^{\frac 12} J (\f(\xi)\big))\\
&=(h( \xi)h(\bar\xi))^{\frac 12}\,J(\U(\bar\xi) J(\f( \xi)))\\
&=J\U(\bar \xi) J(\f( \xi))\\
&=J((\U^\#(\xi))^* J(\f( \xi)))\\
&=(\J(\U^\#)^* \J\f)( \xi).
\end{align*}
Note the use of \eqref{11111LL11ll} above. This proves (a).

Note that 
 $\M_\U\J^{\#}$ is antilinear and isometric on $\mathscr{L}^2(\mu, \h)$. To prove
that $\M_{\U} \J^{\#}$ is a conjugation (and thus complete the proof of (b)), Lemma \ref{csu}  says that we just need to check the identity 
 $$\J^{\#}\M_\U \J^{\#} = \M_{\U}^{*}.$$
  By (a) this is equivalent to  $J\U(\bar \xi)J=\U^*(\xi)$ since, by \eqref{MuuMuu**}, $(\M_\U^*\f)(\xi)=\U^*(\xi)\f(\xi)$. 
  
  Statement (c) follows from the fact that $\M_{\U} \J^{\#}$ is a conjugation on $\mathscr{L}^2(\mu, \h)$, and so $(\M_{\U} \J^{\#})( \M_{\U} \J^{\#}) = I$, along with  the fact $\M_{\U} \M_{\U^{*}} = \M_{\U^{*}} \M_{\U} = I$ (since $\U(\xi)$ is unitary for $\mu$-almost every $\xi \in \T$).
  
 To see (d), observe that for any $\f \in \mathscr{L}^2(\mu, \h)$,
\begin{align*}
(\M_\U\J^{\#}\M_\xi\f)(\xi)&= \U(\xi)\J^{\#}(\M_\xi\f)(\xi)\\
&= \U(\xi)h( \xi)^{\frac 12}\, J((\M_\xi\f)(\bar \xi))\\
&=h( \xi)^{\frac 12}\,\U(\xi) J(\bar \xi\f(\bar \xi))\\
& =\xi h( \xi)^{\frac 12}\, \U(\xi)J(\f(\bar \xi))
\end{align*}
while
\begin{align*}
(\M_\xi\M_\U\J^{\#}\f)(\xi)&=   \xi(\M_\U\J^{\#}\f)(\xi)\\
&   = \xi\U(\xi)(\J^{\#}\f)(\xi) \\
&  =\xi\U(\xi)h( \xi)^{\frac 12}\, J(\f(\bar \xi))\\
&= \xi h( \xi)^{\frac 12}\, \U(\xi)J(\f(\bar \xi)),
\end{align*}
which completes the proof of (d).
\end{proof}

\section{Conjugations and bilateral shifts}\label{bilarererer}

Many interesting, and naturally occurring,  unitary operators are bilateral shifts. Examples include (i) the translation operator $(U f)(x) = f(x - 1)$ on $L^2(\R)$; (ii) the  dilation operator $(U f)(x) = \sqrt{2} f(2 x)$ on $L^2(\R)$, (iii) the Fourier transform on $L^2(\R)$, (iv) the Hilbert transform on $L^2(\R)$, and (v)  the special class of multiplication operators $U f = \psi f$ on $L^2(m, \T)$, where $\psi$ is an inner function.  We refer the reader to  \cite[Example 6.3]{JMMPWR_OT28} to see the bilateral nature of each of these operators is worked out carefully. This section gives an initial  description of $\mathscr{C}_c(U)$ for this class of operators, along with  the important fact that $\mathscr{C}_{c}(U) \not = \varnothing$.  Another, more concrete,  description will be discussed in the next section. Let us begin with a precise definition of    the term  ``bilateral shift''.

\begin{Definition}\label{lafdhgjdfpsgszhnigbv888}
A unitary operator $U$ on $\h$ is a {\em bilateral shift} if there is a subspace $\mathcal{M} \subset \h$ for which
\begin{enumerate}
\item $U^{n} \mathcal{M} \perp \mathcal{M}$ for all $n \in \Z \setminus \{0\}$;
\item ${\displaystyle \h = \bigoplus_{n = -\infty}^{\infty} U^{n} \mathcal{M}}$.
\end{enumerate}
In the above, note that $U^{-1} = U^{*}$. The subspace $\mathcal{M}$ is called an {\em associated wandering subspace} for the bilateral shift $U$.  Of course there is {\em the} bilateral shift $M_{\xi}$ on $L^2(m, \T)$ discussed earlier where a wandering subspace $\mathcal{M}$ can be taken to be the constant functions. 
\end{Definition}

Though the wandering subspace $\mathcal{M}$ in Definition \ref{lafdhgjdfpsgszhnigbv888}  is not unique, its dimension is \cite{MR152896}. The term ``bilateral shift ''  comes from the fact that since
\begin{equation}\label{e15}\mathcal{H}=\bigoplus_{n=-\infty}^{\infty}\, U^n \mathcal{M},
\end{equation}
every $\vec{x} \in \h$ can be uniquely represented as
$$\vec{x} = \sum_{n = -\infty}^{\infty} U^n \vec{x}_n, \; \mbox{where $\vec{x}_n \in \mathcal{M}$ for all $n \in \Z$}.$$
This allows us to define a natural unitary operator
 \begin{equation}\label{WWWWW}
W: \h \to \mathscr{L}^2(m,\m), \quad  W\Big(\bigoplus_{n = -\infty}^{\infty}U^n \vec{x}_n\Big)=\sum_{n = -\infty}^{\infty} \vec{x}_n \xi^n.\end{equation}
 Moreover, thanks to \eqref{e15} and \eqref{WWWWW},
 $WUW^{*}=\M_\xi,$
 where
  $\M_{\xi}$ is the bilateral shift from \eqref{e1.5} defined on $\mathscr{L}^2(m, \mathcal{M})$ by 
  $$\M_{\xi} \f(\xi) = \xi \f(\xi), \; \mbox{where} \; \; 
  \f(\xi) = \sum_{n = -\infty}^{\infty} \vec{x}_{n} \xi^n \in \mathscr{L}^2(m, \mathcal{M}).$$

For a bilateral shift $U$ on $\h$ we wish to describe $\mathscr{C}_c(U)$.
Since $W^{*} U W = \M_{\xi}$ on $\mathscr{L}^2(m, \mathcal{M})$, we know  that for any $C \in \mathscr{C}_{c}(U)$ the mapping 
$$\CCC=W^{*}CW$$  is a conjugation on $\mathscr{L}^2(m,\m)$ such that $\CCC\mathbf{M}_\xi\CCC=\mathbf{M}_{ \xi}$.
The following result from  \cite[Theorem 4.3]{MR4169409} describes $\CCC$.

%

\begin{Theorem}\label{jjHHbbhGHJKK}
For $\M_{\xi}$ on $\mathscr{L}^2(m, \mathcal{M})$ we have the following. 
\begin{enumerate}
\item $\mathscr{C}_{c}(\M_{\xi}) \not = \varnothing$.
\item Fix a conjugation $J$ on $\mathcal{M}$. For a conjugation  $\CCC$ on  $\mathscr{L}^2(m,\m)$,
 the following are equivalent.
 \begin{enumerate}
 \item  $\CCC \in \mathscr{C}_{c}(\M_{\xi})$;
      \item There is $\mathbf{U}\in \mathscr{L}^\infty(m,\mathcal{B}(\m))$ such that $\mathbf{U}(\xi)$ is unitary for almost every $\xi \in \T$, $\M_{\U}$ is $\J^\#$--symmetric, and $\CCC= \M_{\mathbf{U}}{\J}^\#={\J}^\#\M_{\mathbf{U}^*}$.
      \end{enumerate}
      \end{enumerate}
\end{Theorem}
Note that $m=m^c$ and the definition of $\J^\#$  from \eqref{e1.7} coincides with the one appearing in  \cite{MR4169409}.

Theorem \ref{jjHHbbhGHJKK} yields a description of $\mathscr{C}_c(U)$ when $U$ is a bilateral shift.

\begin{Corollary}\label{cococlclcucuc}
Let $U$ be a unitary bilateral shift on $\h$ with an associated wandering subspace $\mathcal{M}$. The we have the following. 
\begin{enumerate}
\item $\mathscr{C}_{c}(U) \not = \varnothing$.
\item For a conjugation $C$ on $\h$, the following are equivalent:
\begin{enumerate}
\item $C \in \mathscr{C}_s(U)$;
\item $\CCC = WC W^{*}$ is a conjugation on $\mathscr{L}^2(m, \mathcal{M})$ that satisfies any of the equivalent conditions of Theorem \ref{jjHHbbhGHJKK}.
\end{enumerate}
\end{enumerate}
\end{Corollary}

Perhaps one might be a bit unsatisfied with the somewhat vague nature of our current description of $\mathscr{C}_{c}(U)$ for a bilateral shift $U$. The next section will give a much more concrete characterization.

\section{Unitary multiplication operators on $L^2(m, \T)$}\label{innerfunctions}

As discussed in  \cite[Example 5.16]{JMMPWR_OT28} there is  model for any bilateral shift $U$ on $\h$ (recall Definition \ref{lafdhgjdfpsgszhnigbv888}) as the multiplication operator $M_{\psi}$ on $L^2 = L^2(m, \T)$, where $\psi$ is an inner function whose degree is that of the dimension of any wandering subspace for $U$. In this section, we give a concrete description of $\mathscr{C}_c(M_{\psi})$. If $J^\#$ is the conjugation on $L^2$ defined by 
\begin{equation}\label{JJJJJJvvv}
 (J^\# f)(\xi) =f^\#(\xi)= \overline{f(\bar \xi)},
 \end{equation}
  and $C \in \mathscr{C}_c(M_{\psi})$, then $C J^\#$ is a unitary operator on $L^2$ for which 
  $$(C J^\#)M_{\psi} = M_{\psi} (C J^\#).$$ This trick was used in several places \cite{MR4083641, MR4169409, MR4502709}. The bounded operators on $L^2$ which commute with  $M_{\psi}$, i.e., the {\em commutant} of $M_{\psi}$,  were described in \cite[Theorem 7.3]{MPRCOUI}.


Recall the known fact (see for example \cite[Proposition 5.17]{JMMPWR_OT28}) that for an inner
 function $\psi$ we have the following orthogonal decomposition for $L^2$, namely, 
\begin{equation}\label{l2sum}L^2 = \bigoplus _{n = -\infty}^{\infty} \psi^{n} \K_{\psi},\end{equation}
where $\K_{\psi} := H^2 \cap (\psi H^2)^{\perp}$ is the model space associated with $\psi$ (see \cite{MR3526203} for a review of model spaces). In other words, $\K_{\psi}$ is a wandering subspace (as in Definition \ref{lafdhgjdfpsgszhnigbv888}) for the multiplication operator $M_{\psi}$.

Let us set up some notation to be used below. For an inner function $\psi$, let
$$N := \operatorname{dim} \K_{\psi} \in \N \cup \{\infty\}$$ and 
 $\{h_j\}_{1 \leq j \leq N}$ be a fixed orthonormal basis for  $\K_\psi$. There are several ``natural'' orthonormal bases one can choose \cite[Ch.~5]{MR3526203}. Observe that $N$ is finite if and only if $\psi$ is a finite Blaschke product with $N$ zeros, repeated according to multiplicity \cite[Prop.~5.19]{MR3526203}. Also define 
$$\bigoplus_{1 \leq j \leq N} L^2 = L^2 \oplus L^2 \oplus  \cdots \oplus L^2.$$
The norm of an  $\f = [f_{j}]_{1 \leq j \leq N}^{t}$ of $ \bigoplus_{1 \leq j \leq N} L^2$ is 
$$\|\f\| := \Big(\sum_{1 \leq j \leq N} \|f_j\|^{2}_{L^2}\Big)^{\frac{1}{2}}.$$
When $N = \infty$, we need to assume that the sum defining $\|\f\|$ above is finite. 
Furthermore, the operator $\bigoplus_{1 \leq j \leq N} M_{\xi}$ (called the {\em inflation} of the bilateral shift $M_{\xi}$ on $L^2$) is given by 
$$\Big(\bigoplus_{1 \leq j \leq N} M_{\xi}\Big) \f(\xi) = \xi \f(\xi) = [\xi f_{j}(\xi)]^{t}_{1 \leq j \leq N}.$$
We also define 
$$\ell^{2}_{N} := \Big\{\vec{x} = [x_j]^{t}_{1 \leq j \leq N}, x_j \in \C: \|\vec{x}\|_{\ell^{2}_{N}} = \Big(\sum_{1 \leq j \leq N} |x_j|^2\Big)^{\frac{1}{2}} < \infty \Big\}.$$ When $N = \infty$, this is the familiar sequence space $\ell^2$. Finally, observe that 
\begin{equation}\label{kk*8S8s8s}
\bigoplus_{1 \leq j \leq N} M_{\xi} \cong \M_{\xi}|_{\mathscr{L}^2(m, \ell^{2}_{N})}.
\end{equation}
 As a consequence, using the discussion from \S \ref{section3}, note that 
\begin{equation}\label{MphisihibibC}
\mathscr{C}_{c}(M_{\psi}) \not = \varnothing.
\end{equation}
We will actually describe $\mathscr{C}_{c}(M_{\psi})$ below. 

From \cite{MPRCOUI} we have the unitary operator
\begin{equation}\label{forW}
  W\colon L^2 \to \bigoplus_{1 \leq j \leq N}L^2, \quad Wf=[f_j]_{1 \leq j \leq N}^t,
\end{equation}
where 
$$f=\sum\limits_{1 \leq j \leq N}h_j \cdot  (f_j \circ \psi)$$ is a unique decomposition given by \cite[Lemma 7.3]{MPRCOUI}. Note that 
\begin{equation}\label{wwYYWYYWwww}
f_j=\sum\limits_{m=-\infty}^{\infty}a_{mj}\xi^m
\end{equation} and the coefficients $a_{mj}$ arise from the decomposition from \eqref{l2sum} which yields the unique decomposition
\begin{equation}\label{zzzxxZZXX}
f=\sum\limits_{1 \leq j \leq N}h_j\sum\limits_{m=-\infty}^{\infty}a_{mj}\psi^m.
\end{equation}

Also recall from  \cite[Thm.~7.3]{MPRCOUI} that 
\begin{equation*}
  W^*[k_j]_{1 \leq j \leq N}^t=\sum\limits_{1 \leq j \leq N}h_j \cdot  (k_j\circ \psi).
\end{equation*}

Let $J$ and $J^\#$  denote the standard conjugations on $L^2$ defined by  
$$Jf(z)=\overline{f(\xi)} \quad \mbox{and} \quad (J^\# f)(\xi)=f^\#(\xi): = \overline{f(\bar \xi)}.$$
 For our inner function $\psi$, observe that  $\psi^\# = J^\# \psi$ is also inner.
\begin{Proposition}
For an inner function $\psi$ we have the following. 
\begin{enumerate}
  \item $J^{\#}\K_\psi = \K_{\psi^{\#}}$.
  \item If $\{h_j\}_{1 \leq j \leq N}$ is an orthonormal basis for $\K_\psi$ then $\{h_j^{\#}\}_{1 \leq j \leq N}$ is an orthonormal basis for $\K_{\psi^{\#}}$.
\end{enumerate}
\end{Proposition}

\begin{proof}
Part (a) was shown in \cite[Lemma 4.4]{MR4083641} while part  (b) is a consequence of the facts that conjugations preserve  orthonormality (recall \eqref{CCCCCCcccc}).
\end{proof}

Let $W_{\#}$ be  the  unitary operator from \eqref{forW}, where the inner function $\psi$ is replaced by   $\psi^{\#}$ and orthonormal basis and the orthonormal basis $\{h_j\}_{1 \leq j \leq N}$ is replaced by the orthonormal basis $\{h_j^{\#}\}_{1 \leq j \leq N}$, i.e.,
 $$W_{\#}g=[g_j]_{1 \leq j \leq N}^t, \; \mbox{where} \;  g = \sum_{1 \leq j \leq N}h_j^{\#} \cdot (g_j\circ \psi^{\#}).$$

There are the two natural conjugations $\J$ and $\J^{\#}$ on $\bigoplus_{1 \leq j \leq N} L^{2}$ defined for each  $\vec{F} \in \bigoplus_{1 \leq j \leq N} L^{2}$, $\vec{F}= [f_j]^{t}_{1 \leq j \leq N} $, by 
\begin{align*}\J\vec{F}&= [  \bar f_j]^{t}_{1 \leq j \leq N}  =: \bar{{\vec{F}}} \quad \mbox{and} \quad \J^\#\vec{F} =[  f_j^\#]^{t}_{1 \leq j \leq N}=:{\vec{F}}^\#.
\end{align*}

\begin{Proposition}\label{prop1.3}Let $\psi$ be an inner function and $\{h_j\}_{1 \leq j \leq N}$ be an orthonormal basis for $\K_\psi$.
 Then we have the following. 
\begin{enumerate}
\item If ${\displaystyle f=\sum\limits_{1 \leq j \leq N}h_j \cdot  (f_j\circ \psi)}$ then ${\displaystyle f^\#=J^{\#}f=\sum\limits_{1 \leq j \leq N}h_j^{\#}\cdot (f_j^{\#}\circ \psi^{\#})}$.
\item
$W_{\#}J^{\#}W^*=\J^{\#}$.
\end{enumerate}
\end{Proposition}

\begin{proof}
%

Let $f\in L^2$ and observe from \eqref{wwYYWYYWwww} and \eqref{zzzxxZZXX} that  \[J^{\#}f=\sum\limits_{j=1}^{\infty}h_j^{\#}\sum \overline{a_{mj}}(\psi^{\#})^m\quad \text{and}\quad f_j^{\#}=\sum\limits_{m=-\infty}^{\infty}\overline{a_{mj}}\xi^m.\] Hence 
$$J^{\#}f=\sum\limits_{j=1}^{\infty}h_j^{\#} \cdot (f_j^{\#}\circ \psi^{\#}),$$ 
which proves (a). The above also yields  
$$W_{\#}J^{\#}W^*[f_j]_{1 \leq j \leq N}^t=W_{\#}J^{\#}f=[f_j^{\#}]_{1 \leq j \leq N}^t=
\J^{\#}[f_j]_{1 \leq j \leq N}^t,$$
which proves (b).
\end{proof}

\begin{Theorem}\label{089foidjgigifghjhjh77766}
Suppose that $\psi$ is inner and  $\{h_j\}_{1 \leq j \leq N}$ is an orthonormal basis for $\K_{\psi}$. Then we have the following, 
\begin{enumerate}
\item $\mathscr{C}_{c}(M_{\psi}) \not = \varnothing$. 
\item $C \in \mathscr{C}_{c}(M_{\psi})$ if and only if there is a $\Phi = [\phi_{i j}]_{1 \leq i, j \leq N} \in \mathscr{L}^{\infty}(m, \ell^2_{N})$ such that 
\begin{equation}\label{phishishishis}
 \Phi(\xi)=\Phi(\bar \xi) \; \mbox{and} \; \Phi^{*}(\xi) \Phi(\xi) = I
 \end{equation} almost everywhere on $\T$    and
\begin{equation}\label{cfor} C f = \sum_{1 \leq j \leq N} ( f^\#_j \circ \psi) \sum_{1 \leq k \leq N} h_k  \cdot (\phi_{k, j} \circ \psi),\end{equation}
for all ${\displaystyle f = \sum_{1 \leq j \leq N} h_j \cdot (f_j \circ \psi)\in L^2}$.
\end{enumerate}
\end{Theorem}
\begin{proof}
Statement (a) follows from \eqref{MphisihibibC}. To prove (b), observe that since $W$ is a unitary operator, then $\widetilde C:=WCW^*$ is a conjugation on  $\bigoplus_{1 \leq j \leq N} L^2$ (Lemma \ref{lem1.2}). If $C M_{\psi}  = M_{{\psi}}C$, it follows from 
\cite[Theorem 7.2(c)]{MPRCOUI} that 
 $$\widetilde C\Big( \bigoplus_{1 \leq j \leq N}M_{\xi}\Big) = \Big(\bigoplus_{1 \leq j \leq N}M_{\xi}\Big)\widetilde C .$$
Since the operator $\bigoplus_{1 \leq j \leq N}M_{\xi}$ on $\bigoplus_{j\geq 1}L^2$ is unitary equivalent to $\M_{\xi}$ on $\mathscr{L}^2(m,\ell^2_N)$ (recall \eqref{kk*8S8s8s}), Theorem \ref{jjHHbbhGHJKK} says there is a $\Phi = [\phi_{i j}]_{1 \leq i, j \leq N} \in \mathscr{L}^{\infty}(m, \ell^2_{N})$ such that $\Phi(\xi)$ is unitary for almost every $\xi \in \T$, 
$\M_{\Phi}$ is $\J^\#$--symmetric, and 
$$\widetilde C=\M_\Phi \J^\#.$$ The unitary property gives  $\Phi^*(\xi)\Phi(\xi)=I$ and the $\J^\#$--symmetry property gives $\Phi(\xi)=\Phi(\bar \xi)$ almost everywhere on $\T$. So far, we have shown that if $C$ is a conjugation which commutes with $M_{\psi}$, then $W C W^{*} = \M_{\Phi} \J^{\#}$, where $\Phi$ satisfies the two properties from \eqref{phishishishis}. Conversely suppose that $\Phi = [\phi_{i j}]_{1 \leq i, j \leq N} \in \mathscr{L}^{\infty}(m, \ell^2_{N})$ satisfies the two conditions from \eqref{phishishishis}. The second condition will show that $\J^{\#} \M_{\Phi} \J^{\#} = \M_{\Phi}^{*}$ and combining this with the first condition will show that $\Phi$ is unitary valued almost everywhere. The second property, along with Proposition \ref{p3.4} will show that $\M_{\Phi} \J^{\#}$ is a conjugation and belongs to $\mathscr{C}_{c}(\M_{\xi})$. By the discussion above, this says that $W^{*} (\M_{\Phi} \J^{\#}) W \in \mathscr{C}_{c}(M_{\psi})$. 

Applying  Proposition \ref{prop1.3} we can verify the formula \eqref{cfor}. Indeed, for each 
$$f=\sum\limits_{1 \leq j \leq N}h_j \cdot (f_j\circ \psi) \in L^2$$ we have  
\begin{align*}
C f & = W^{*} {\M}_{\Phi} \J^{\#} W f\\
 & = W^{*} {\M}_{\Phi} W_\# J^{\#} \Big(\sum\limits_{1 \leq j \leq N}h_j \cdot (f_j\circ \psi)\Big)\\
 & = W^{*} {\M}_{\Phi} W_\#\Big(\sum\limits_{1 \leq j \leq N}h_j^{\#} \cdot (f_j^{\#}\circ \psi^{\#})\Big)\\
& = W^{*} [\phi_{ij} ]_{1 \leq i, j \leq N}[ f^\#_j]_{1 \leq j \leq N}^{t}\\
& = W^{*} \Big[\sum_{1 \leq i, j \leq N} \phi_{1j}  f^\#_j, \sum_{1 \leq i, j \leq N} \phi_{2j}  f^\#_j, \sum_{1 \leq i, j \leq N} \phi_{3j}  f^\#_j, \ldots\Big]^{t}\\
& = h_1\cdot \Big(\sum_{1 \leq i, j \leq N} \phi_{1j}  f^\#_j\Big)\circ \psi + h_2 \cdot \Big(\sum_{1 \leq i, j \leq N} \phi_{2j}  f^\#_j\Big) \circ \psi + \ldots\\
&=  (f^\#_{1} \circ \psi) \cdot  \Big(\sum_{1 \leq i, j \leq N} h_{j} (\phi_{j1} \circ \psi\Big) +  (f^\#_{2} \circ \psi) \cdot  \Big(\sum_{1 \leq i, j \leq N} h_{j} (\phi_{j2} \circ \psi\Big) + \cdots
\end{align*}
and this completes the proof. 
\end{proof}

\begin{Example}\label{sdjkfksdlfkHHGGHVGBVGHBGH}
Consider the inner function $\psi(z) = z$. Here the associated unitary operator $M_{\psi}$ is merely {\em the} bilateral shift $M_{\xi}$ on $L^2$. In this case, $\K_{\psi} = \C$ (the constant functions). Moreover, $\psi^{\#}(z) = z$ and the expansions from Proposition \ref{prop1.3} are the standard Fourier expansions of an $f \in L^2$. Theorem \ref{089foidjgigifghjhjh77766} says that any $C \in \mathscr{C}_{c}(M_{\xi})$ takes the form $$(C f)(\xi) = u(\xi) \overline{f(\bar{\xi})}$$ for some $u \in L^{\infty}$ that is unimodular   and satisfies $u(\xi) = u(\bar \xi)$ almost everywhere on $\T$. 
\end{Example}

\begin{Example} Consider the inner function 
 $\psi(z) = z^2$ as in \cite[Example 7.7]{MPRCOUI}. Then $\K_{\psi} = \operatorname{span} \{1, z\} = \{h_1, h_2\}$. Furthermore, using the notation from this section, 
$$f(\xi) = h_{1}(\xi) f_{1}(\xi^2) + h_{2}(\xi) f_{2}(\xi^2) = f_{1}(\xi^2) + \xi f_{2}(\xi^2),$$ where
$$f_1(\xi) = \sum_{j = -\infty}^\infty \widehat{f}(2 j) \xi^j \quad \mbox{and} \quad f_{2}(\xi) = \sum^\infty_{j = -\infty} \widehat{f}(2 j + 1) \xi^j.$$ From here, one can check
(Theorem \ref{089foidjgigifghjhjh77766})  that every $C \in \mathscr{C}_s(M_{\xi^2})$ takes the form
$$(C f)(\xi) = { f_{1}^\#(\xi^2)} (\phi_{11}(\xi^2) + \xi \phi_{21}(\xi^2)) + { f_{2}^\#(\xi^2)} (\phi_{12}(\xi^2) + \xi \phi_{22}(\xi^2)),$$
where $\phi_{ij}$ are bounded measurable functions on $\T$ for which
\begin{equation}\label{ex87} \begin{bmatrix}
\overline{ \phi_{11}(\xi)} & \overline{ \phi_{21}(\xi)}\\
\overline{ \phi_{12}(\xi)} & \overline{ \phi_{22}(\xi)}
\end{bmatrix} \begin{bmatrix}
\phi_{11}(\xi) & \phi_{12}(\xi)\\
\phi_{21}(\xi) & \phi_{22}(\xi)
\end{bmatrix}=
\begin{bmatrix}
1 & 0\\
0 & 1
\end{bmatrix}.\end{equation}%
and $\varphi_{ij}(\bar \xi)=\varphi_{ij}(\xi)$, $i,j=1,2$,
for almost every $\xi \in \T$. Condition \eqref{ex87} is equivalent to the conditions 
$$|\phi_{11}(\xi)|^2+|\phi_{21}(\xi)|^2= 1,$$
$$|\phi_{12}(\xi)|^2+|\phi_{22}(\xi)|^2= 1,$$
$$\overline{ \phi_{11}(\xi)}\phi_{12}(\xi)+\overline{\phi_{21}(\xi)}\phi_{22}(\xi)=0.$$
Fix the convention that 
$t=\operatorname{Arg}(\xi)\in (-\pi,\pi]$ and that $s(t),\alpha(t),\beta(t),\gamma(t),\delta(t)$ are any $2\pi$--periodic real-valued  bounded measurable functions. Considering the moduli of the functions above, we obtain 
$$0 \leq s(t) \leq 1,$$
$$\phi_{11}(\xi)=e^{i\alpha(t)}s(t),$$
$$\phi_{12}=e^{i\beta(t)}\sqrt{1-s^2(t)},$$
$$\phi_{21}(\xi)=e^{i\gamma(t)}\sqrt{1-s^2(t)},$$
$$\phi_{22}(\xi)=e^{i\delta(t)}s(t).$$
As to the arguments of the functions above, we obtain 
$$\delta(t)=\beta(t))+\gamma(t)-\alpha(t)-\pi.$$ Incorporating the conditions $\varphi_{ij}(\bar \xi)=\varphi_{ij}(\xi)$, $i,j=1,2$,  we  obtain 
\begin{equation}\label{ex872} \begin{bmatrix}
\phi_{11}(\xi) & \phi_{12}(\xi)\\
\phi_{21}(\xi) & \phi_{22}(\xi)
\end{bmatrix}= \begin{bmatrix}
e^{i\alpha(|t|)}s(|t|) & e^{i\beta(|t|)}\sqrt{1-s^2(|t|)}\\
e^{i\gamma(|t|)}\sqrt{1-s^2(|t|)} & -e^{i(\beta(|t|)+\gamma(|t|)-\alpha(|t|))}s(|t|)
\end{bmatrix}.\end{equation}%
 Finally,  every conjugation $C \in \mathscr{C}_s(M_{\xi^2})$ must take the form
\begin{multline*}
(C f)(\xi) = { f^\#_{1}(\xi^2)}\Big( e^{i\alpha(2|t|)}s(2|t|) + \xi e^{{i}(\gamma(2|t|))}\sqrt{1-s^2(2|t|)}\Big)\\ + { f^\#_{2}(\xi^2)}\Big( e^{{i}\beta(2|t|)}\sqrt{1-s^2(2|t|)} -\xi e^{i(\beta(2|t|)+\gamma(2|t|)-\alpha(2|t|))}s(2|t|)\Big),\end{multline*}
where $t=\operatorname{Arg}(\xi)\in (-\pi,\pi]$ and $s(t),\alpha(t),\beta(t), \gamma(t)$ are any $2\pi$--periodic real  bounded measurable functions.
\end{Example}
\begin{Example}
As a specific nontrivial example of a $C \in \mathscr{C}_{c}(M_{\xi^2})$ we can take
$$(C f)(\xi) = { f^\#_{1}(\xi^2)} \big(\sin(2|t|) + \xi \cos(2t)\big) + { f^\#_{2}(\xi^2)}\big(\cos(2t) - \xi \sin(2|t|)\big),$$
where $t=\operatorname{Arg}(\xi)$ and $s(t)= \sin(t),\alpha(t)\equiv 0,\beta(t)\equiv 0,\gamma(t)\equiv 0$.
For another nontrivial example of a $C \in \mathscr{C}_{c}(M_{\xi^2})$, we set $s(t)\equiv s\in [0,1],\alpha(t)=\lambda t, \lambda\in\R,\beta(t)\equiv 0,\gamma\equiv 0$ to get
$$(C f)(\xi) = { f^\#_{1}(\xi^2)} \big(s e^{i\lambda |t|}+ \xi \sqrt{1-s^2}\big) + { f^\#_{2}(\xi^2)}\big(\sqrt{1-s^2} - \xi e^{i\lambda |t|})\big),$$
where $t=\operatorname{Arg}(\xi)$.

\end{Example}




\section{Conjugations via the spectral theorem}\label{spectraltheoremre}

In this section we use the multiplicity theory for unitary operators \cite{ConwayFA, MR0045309} to describe $\mathscr{C}_{c}(U)$. We also prove that $\mathscr{C}_{c}(U) \not = \varnothing$ if and only if $U \cong U^{*}$ (thus establishing the converse to Lemma \ref{uustar}). We begin with a statement of the spectral multiplicity theory from \cite[p.~ 307, Ch. IX, Theorem 10.20]{ConwayFA}. 

\begin{Theorem}[Spectral Theorem]\label{spectraltheorem}
For a unitary operator $U$ on $\h$, there are mutually singular measures
$\mu_{\infty}, \mu_1, \mu_2, \ldots \in M_{+}(\T)$, along with Hilbert spaces $\h_{\infty}, \h_{1}, \h_2, \ldots$ each with corresponding  $\operatorname{dim} \h_{k} = k$, $k = \infty, 1, 2, 3, \dots$, along with  an isometric  isomorphism
$$\mathcal{I}: \h \to \LL^2_{\h}: =  \mathscr{L}^2(\mu_\infty,\h_\infty)\oplus \mathscr{L}^2(\mu_1,\h_1)\oplus  \mathscr{L}^2(\mu_2,\h_2)\oplus \cdots$$ such that
$\mathcal{I}U \mathcal{I}^{*} $ is equal to the unitary operator 
\begin{equation}\label{MMMMmmmmm}
\M_{\xi}^{(\infty)} \oplus \M_{\xi}^{(1)} \oplus \M_{\xi}^{(2)} \oplus \cdots,
\end{equation}
where for $i = \infty, 1, 2, 3, \ldots$,
$$\M_{\xi}^{(i)}: \mathscr{L}^2(\mu_{i}, \h_{i}) \to  \mathscr{L}^2(\mu_{i}, \h_{i}), \quad (\M_{\xi}^{(i)} \f)(\xi) = \xi \f(\xi).$$
\end{Theorem}

\begin{Remark}\label{r5.2} Let $U$ be a unitary operator with a spectral measure $E(\cdot)$. As previously observed in Proposition \ref{4}(a), $E^c(\cdot)$ is a spectral measure for $U^*$. In \cite[Theorem 8.1]{MPRCOUI}, the measures $\mu_\infty,\mu_1,\mu_2,\dots$ from Theorem \ref{spectraltheorem} were constructed using the spectral measure $E(\cdot)$. Therefore, the appropriate measures for operator $U^*$ are $\mu^c_\infty,\mu^c_1,\mu^c_2,\dots$.
\end{Remark}

We now consider unitary operators with  commuting conjugations. Recall from Proposition \ref{p3.1} that not all unitary operators have commuting conjugations. In the next result, we re-emphasize this observation in terms of the multiplicity theory from Theorem \ref{spectraltheorem}.

\begin{Theorem}\label{th1.2b}
Let $U$  be a unitary operator on $\h$ with the multiplicity representation of $U$ given by the  mutually singular measures $\mu_\infty,\mu_1,\mu_2,\dots$ as in Theorem \ref{spectraltheorem}. 
If $\mathscr{C}_{c}(U) \not = \varnothing$, then $\mu_i^c\ll\mu_i$ for all $i=\infty,1,2,\dots$.
\end{Theorem}
\begin{proof}
Lemma \ref{uustar} says that if $\mathscr{C}_{c}(U) \not = \varnothing$, then  $U \cong U^{*}$. Hence, by Remark \ref{r5.2} and \cite[p.~305, Theorem IX 10.16]{ConwayFA}, the measures $\mu_i$ and $\mu^c_i$ are mutually absolutely continuous for  all $i=\infty,1,2,\dots$.
\end{proof}

We now arrive at the description of $\mathscr{C}_{c}(U)$ in terms of the parameters of the spectral theorem. 

\begin{Theorem}\label{th1.2a}
Let $U$ be a unitary operator and $C$ be a conjugation on $\h$. 
With the notation as in Theorem \ref{spectraltheorem}, assuming that $\mu_i^c\ll\mu_i$ for $i=\infty,1,2,\dots$, the following are equivalent
\begin{enumerate}
  \item $C \in \mathscr{C}_{c}(U)$;
      \item  For each $i=\infty,1,2,\dots$, there are conjugations $\CC^{i}\in \mathscr{A}\! \mathcal{B}((\mathscr{L}^{2}(\mu_i,\h_i))$ such that $\M_\xi^{(i)}$ is $\CC^{i}$--commuting and 
      \begin{equation}\label{eqm1}
        C={\II}^{*}\Big(\bigoplus \CC^i\Big)\II;
      \end{equation}%
      \item  For each $i=\infty,1,2,\dots$ and  any conjugation $J^{(i)}$ on $\h_i$, there is a unitary operator valued function $\U^{(i)}\in \mathscr{L}^{\infty}(\mu_i,\mathcal{B}(\h_i))$ such that 
      $$J^{(i)}\U^{(i)}(\xi)J^{(i)}=\U^{(i)}(\xi)^{\#}$$ for $\mu_i$ almost every $\xi \in \T$ and
      \begin{equation}\label{eqm3}
        C={\II}^{*}\Big(\bigoplus\U^{(i)}{\J^{\#}}^{(i)}\Big)\II={\II}^{*}\Big(\bigoplus\U^{(i)}\Big)\Big(\bigoplus {\J^{\#}}^{(i)}\Big)\II.
      \end{equation}
\end{enumerate}
\end{Theorem}
\begin{proof}
To show (a) $\Longrightarrow$ (c), 
 let $$\widetilde{\M}_\xi := \mathcal{I}V\mathcal{I}^{*}\in \mathcal{B}(\LL^2_{\h})$$ and define the conjugation $\widetilde{\!\CC}=\mathcal{I}C\mathcal{I}^{*}$ (note the use of Lemma \ref{lem1.2}). Then
\begin{equation}\label{e1.8a}
  \widetilde{\!\CC}\widetilde{\M}_\xi\,\widetilde{\!\CC}=\widetilde{\M}_{ \xi}.
\end{equation}

Let $J^{(i)}$ be any a conjugation on $\h_i$. Since $\mu^c_i\ll\mu_i$ for $i=\infty,1,2,\dots$, let 
$$h_i = \frac{d \mu^{c}_{i}}{d \mu_{i}}$$ and define the map  ${\J^{\#}}^{(i)}$ on $\mathscr{L}^2(\mu_i,\h_i)$ by
\begin{equation*}
  ({\J^{\#}}^{(i)}\f_i)(\xi)=  h_i( \xi)^{\frac 12}\, J^{(i)}(\f_i(\bar \xi))
\end{equation*} 
for $\mu_i$ almost every $\xi \in \T$ and $\f_i\in \mathscr{L}^2(\mu_i,\h_i)$. By Proposition \ref{p3.2a}, each of the above maps defines a conjugation on $\mathscr{L}^2(\mu_i, \h_{i})$ which satisfies
$${\J^{\#}}^{(i)} \M_\xi^{(i)} {\J^{\#}}^{(i)}=\M_\xi^{(i)}.$$
Use these conjugations to define the conjugation $\widetilde{\!\J\,}^{\#}=\bigoplus {\J^{\#}}^{(i)}$ on $\LL^2_{\h}$
and observe  that 
$$ \widetilde{\!\J\,}^{\#}\,\widetilde{\M}_\xi\,\widetilde{\!\J\,}^{\#}=\widetilde{\M}_{ \xi}.$$
Moreover, 
\begin{equation*}
  \widetilde{\M}_\xi\ \widetilde{\!\CC}\ \widetilde{\!\J\,}^{\#}=\widetilde{\!\CC}\ \widetilde{\M}_{\xi}\ \widetilde{\!\J\,}^{\#}=\widetilde{\!\CC}\ \widetilde{\!\J\,}^{\#}\ \widetilde{\M}_\xi.
\end{equation*}
The spectral theorem applied to $\widetilde{\M_{\xi}}$ also yields the commutant  \cite[p.~ 307, Theorem 10.20]{ConwayFA}, namely
 there are 
\begin{equation*}
  \U^{(i)}\in \mathscr{L}^{\infty}(\mu_i,\mathcal{B}(\h_i)), \quad  i=\infty,1,2,\dots,
\end{equation*}
such that
\begin{equation*}
  \widetilde{\!\CC}\, \widetilde{\!\J\,}^{\#}=\bigoplus\widetilde{\M}_{\U^{(i)}}=\widetilde{\M}_{\U^{(\infty)}}\oplus \widetilde{\M}_{\U^{(1)}} \oplus \widetilde{\M}_{\U^{(2)}} \oplus \cdots.
\end{equation*}
Since $\widetilde{\!\CC}\, \widetilde{\!\J\,}^{\#}$ is unitary, it follows that  $\widetilde{\M}_{\U^{(i)}}$ is also unitary and consequently $\U^{(i)}$ is a operator valued operator function such that  $\U^{(i)}(\xi)$ is unitary for $\mu_i$ almost every $\xi \in \T$. Therefore,
\begin{equation*}
  \widetilde{\!\CC}=\Big(\bigoplus \M_{\U^{(i)}}\Big)\Big(\bigoplus {\J^{\#}}^{(i)}\Big)=\bigoplus \M_{\U^{(i)}}{\J^{\#}}^{(i)}.
\end{equation*}
Since $\widetilde{\!\CC}|_{\mathscr{L}^2(\mu_i,\h_i)}$ is a conjugation, it follows that
   $$J^{(i)}\U^{(i)}(\xi)J^{(i)}= (\U^{(i)}( \xi))^\#$$  for $\mu_i$ almost every $\xi \in \T$ (Proposition \ref{p3.4}). This completes the proof of (a) $\Longrightarrow$ (c) .

To prove (c) $\Longrightarrow$ (b), it is enough to take $\CC^{(i)}=\M_{\U^{(i)}}{\J^{\#}}^{(i)}$. The remaining implication (b) $\Longrightarrow$ (a) is trivial.
\end{proof}

The following yields  the converse of Lemma \ref{uustar} and thus completes the criterion as to when $\mathscr{C}_{c}(U) \not = \varnothing$.

\begin{Corollary}\label{uustar1}
 If $U$ is a unitary operator on $\h$ such that $U \cong U^{*}$,  then there is conjugation $C$ on $\h$ such that $C U C = U$.
 \end{Corollary}
\begin{proof} If $U \cong U^{*}$, then, as in the proof of  Theorem \ref{th1.2b}, the measures $\mu_i$ and $\mu^c_i$ are mutually absolutely continuous for all  $i=\infty,1,2,\dots$. Now invoke Theorem  \ref{th1.2a} with any conjugation $J^{(i)}$ on $\h_j$ (and  $\U^{(i)}=I_{\h_i}$) and observe that the conjugation
$$C=\II^{*}\J^{\#} \II=\II^{*}\bigoplus{\J^{\#}}^{(i)} \II$$ commutes with $U$. 
\end{proof}

\section{Examples}


 \begin{Example}
 In Example \ref{sdjkfksdlfkHHGGHVGBVGHBGH} we worked out $\mathscr{C}_{c}(M_{\xi})$ for the bilaterai shift $M_{\xi}$ on $L^2(m, \T)$. 
This  example contains a description of $\mathscr{C}_{c}(U)$ when $U = M_{\xi}$ on a more complicated $L^2(\mu, \T)$ space. Let $g\colon [-1,1]\to [0, \infty)$ be defined piecewise  by
 \[g(t)=\left\{
          \begin{array}{ll}
           \tfrac{3}{2}\,t^2, & \hbox{$t\in [0,1]$;} \\[4pt]
           \frac{5}{2}\,t^4, & \hbox{$t\in[-1,0]$.}
          \end{array}
        \right.
 \]

If $dt$ represents Lebesgue measure on $[-1,1]$, define the following  measures on the Borel subsets $\Omega \subset [-1,1]$ by 
 \begin{align*}\widetilde\mu_1(\Omega)&=\int_\Omega g(t)\,dt,\\
\widetilde\mu_2(\Omega)&=\int_\Omega g(-t)\,dt.\end{align*}
One can verify that 
\begin{align*}
\widetilde{h}(t) & :=\frac{d\tilde\mu_2}{d\tilde\mu_1}(t)=\frac{d\tilde\mu_2}{dt}\big(\frac{d\tilde\mu_1}{dt}\big)^{-1}(t)\\
& =
\left\{
          \begin{array}{ll}
            \tfrac{5}{3} t^2, & \hbox{$t\in [0,1]$;} \\[4pt]
            \tfrac{3}{5} t^{-2}, & \hbox{$t\in[-1,0]$.}
          \end{array}
        \right.
\\
& =\big(\tfrac{5}{3}\big)^{\sgn (t)}\, t^{2\,\sgn(t) }.
\end{align*}
Clearly $\widetilde{h}(t)\cdot\widetilde{h}(-t)=1$ on $[-1, 1]$.
 Now let 
 $$\gamma\colon [-1,1]\to\mathbb{T},  \quad \gamma(t)=\exp(2\pi it)$$
and  check that  
$$\gamma^{-1}(\xi)=  \tfrac{\Arg \xi}{2\pi}, \quad \xi \in \T.$$ Define measures $\mu_1,\mu_2 \in M_{+}(\T)$  on Borel sets $\Omega \subset \T$ by 
$$\mu_k(\Omega)=\widetilde{\mu}_k(\gamma^{-1}(\Omega)),  \quad k=1,2,$$
and observe that
 $\mu_1^c=\mu_2$ and $\mu_2\ll\mu_1$. Moreover, we can write the Radon-Nikodym derivative
\[h(\xi):=\frac{d\mu_2}{d\mu_1}(\xi)=\widetilde{h}(\gamma^{-1}(\xi))=\big(\tfrac{5}{3}\big)^{\sgn (\Arg \xi)}\, (\Arg \xi)^{2\,\sgn(\Arg \xi) }.\]
From here one sees that  $h(\xi)h(\bar \xi)=1$ on $\T$ as demonstrated in Proposition \ref{p3.1a}.
%

%

 Now  consider the space $L^2(\mu_1)$ and, as in \eqref{e1.7a}, define the mapping $J^{\#}$ on $L^2(\mu_1)$ by 
\begin{equation}\label{e1.7ca}
  (J^\# f)(\xi)=  h( \xi)^{\frac 12}\ \overline{f(\overline \xi)}.
\end{equation} Then $J^\#$ is a conjugation on $L^2(\mu_1)$ and  $  J^\# M_\xi J^\#= M_\xi$.
Moreover, Theorem \ref{th1.2a} says that any conjugation $C$ on $L^2(\mu_1)$ such that  $  J^\# M_\xi J^\#= M_\xi$ can be expressed  as $C=uJ^\#$ where $u\in L^\infty(\mu_1)$ is a unimodular function such that $u(\overline \xi)=u(\xi)$ for $\mu_1$ almost every $\xi \in \T$.
\end{Example}

\begin{Example}\label{Foureoirtert1}
Let $\mathcal{F}$ denote the standard Fourier--Plancherel transform on $L^2(\R)$. It is well known that $\mathcal{F}$ is unitary and that 
$\sigma(\mathcal{F})  = \{1, i, -1, -i\}.$ Moreover,  the Hermite functions $\{H_n\}_{n \geq 0}$ form an orthonormal basis for $L^2(\R)$ and $\mathcal{F} H_n = (-i)^n H_n$ for all $n \geq 0$, i.e., the Hermite functions form an eigenbasis for $\mathcal{F}$ \cite[Ch.11]{MR4545809}. A description of $\mathscr{C}_{s}(\mathcal{F})$, the symmetric conjugations, was given in \cite[Example 4.3]{MPRCOUI}. In this example we work out $\mathscr{C}_{c}(\mathcal{F})$, the commuting conjugations for $\mathcal{F}$. We first note that $\mathcal{F} \cong \mathcal{F}^{*}$ (Example \ref{JHKSJDF99}). Thus, $\mathscr{C}_{c}(\mathcal{F}) \not = \varnothing$ (Corollary \ref{uustar1}).

To describe $\mathscr{C}_{c}(\mathcal{F})$, we proceed as follows. Our discussion so far says that 
\begin{equation}\label{e_fur}
L^2(\R) =\mathscr{E}_{1} \oplus \mathscr{E}_{-i} \oplus \mathscr{E}_{-1} \oplus \mathscr{E}_{i},\end{equation}
where $\mathscr{E}_{\alpha} = \ker (\mathcal{F} - \alpha I)$.
Define a conjugation $J$ on $L^2(\R)$ for which $J H_n = H_n$ for all $n \geq 0$ (initially define $J$ on $H_n$ by $J H_n = H_n$ and extend antilinearly to all of $L^2(\R)$). 

If $\ell^2$ is the  classical sequence space 
$$\ell^2 = \Big\{\vec{a} := [a_n]_{n \geq 0}^{t}, a_n \in \C: \|\vec{a}\| := \Big(\sum_{n = 0}^{\infty} |a_n|^2\Big)^{\frac{1}{2}} < \infty\Big\}$$ with the standard orthonormal basis $\{\vec{e}_n\}_{n \geq 0}$, and 
$$V = V_{1} \oplus V_{-i} \oplus V_{-1} \oplus V_{i},$$ where 
$V_{1}$ is the unitary from $\mathscr{E}_{1}$ to $\ell^2$ defined by $V_{1}(H_{4n}) = \vec{e}_n$; 
$V_{-i}$ is the unitary from $\mathscr{E}_{-i}$ to $\ell^2$ defined by $V_{-i}(H_{4 n + 1}) = \vec{e}_n$; 
$V_{-1}$ is the unitary  from $\mathscr{E}_{-1}$ to $\ell^2$ defined by $V_{-1}(H_{4 n + 2}) = \vec{e}_n$; $V_{i}$ is the unitary from $\mathscr{E}_{i}$ to $\ell^2$ defined by $V_{i}(H_{4 n + 3}) = \vec{e}_n$; then $V$ is a unitary operator from $L^2(\R)$ onto $\mathscr{L}^2(\mu, \ell^2)$, where $\mu=\delta_{1}+\delta_{-i}+\delta_{-1}+\delta_{i}$.  

Define a conjugation $\widetilde J$ on $\ell^2$ by $\widetilde J(\vec{e}_n)=\vec{e}_n$ for all $n \geq 0$.
Since $\mu^c\ll\mu$ we can define a conjugation $\widetilde \J^\#$ on $\mathscr{L}^2(\mu, \ell^2)$ such that 
$$(V\mathcal{F}V^{*})\widetilde\J^\#= \widetilde\J^\# (V\mathcal{F}V^{*})$$ by \eqref{e1.7}. In other words, with respect to the orthogonal decomposition 
$$\mathscr{L}^2(\mu, \ell^2) = \mathscr{L}^2(\delta_1, \ell^2) \bigoplus \mathscr{L}^2(\delta_{-i}, \ell^2) \bigoplus \mathscr{L}^2(\delta_{-1}, \ell^2) \bigoplus \mathscr{L}^2(\delta_i, \ell^2),$$
the conjugation $\widetilde\J^{\#}$ can be written in matrix form as 
$$\widetilde\J^\#=
  \begin{bmatrix}
    \widetilde J & 0 & 0 & 0\\
    0 & 0 & 0 & \widetilde J \\
    0 & 0 & \widetilde J & 0 \\
    0 & \widetilde J & 0 & 0 
\end{bmatrix}.
$$
The conjugation $V^{*}\J^\#V$ commutes with $\mathcal{F}$ and  it can be written with respect to the Hermite basis  as
\[\J^\#H_{4n+k}:= W^{*}\widetilde\J^\#W H_{4n+k}=\left\{
                  \begin{array}{ll}
                    H_{4n+k}, & \hbox{$k=0$;} \\
                    H_{4n+k+2}, & \hbox{$k=1$;} \\
                    H_{4n+k}, & \hbox{$k=2$;} \\
                    H_{4n+k-2}, & \hbox{$k=3$.}
                  \end{array}
                \right.
  \]Therefore, the matrix representation of $\J^{\#}$ with respect  to the orthogonal decomposition in \eqref{e_fur} is
  $$\J^\#=
  \begin{bmatrix}
     J & 0 & 0 & 0\\
    0 & 0 & 0 &  J \\
    0 & 0 &  J & 0 \\
    0 &  J & 0 & 0 
  \end{bmatrix}.$$
Moreover, by Theorem \ref{th1.2a}, any conjugation $\widetilde C$ on $\mathscr{L}^2(\mu, \ell^2)$ such that 
$$\widetilde C(V\mathcal{F}V^{*})=(V\mathcal{F}V^{*})\widetilde C$$
can be represented  by the matrix
$$\widetilde C=
  \begin{bmatrix}
    \widetilde U_{1}\widetilde J & 0 & 0 & 0\\
    0 & 0 & 0 & \widetilde U_{-i}\widetilde J \\
    0 & 0 & \widetilde U_{-1}\widetilde J & 0 \\
    0 & \widetilde U_i\widetilde J & 0 & 0 \\
  \end{bmatrix},$$
where $\widetilde U_1,\widetilde U_{-i},\widetilde U_{-1},\widetilde U_{i}$, are unitary operators  on $\ell^2$ and 
$$\widetilde J\widetilde U_1\widetilde J=\widetilde U_{1}^{*}, \;  \widetilde J\widetilde U_{-1}\widetilde J=\widetilde U_{-1}^{*}, \; \widetilde J\widetilde U_{i}\widetilde J =\widetilde U_{-i}^{*}.$$ The first two identities say that the  unitary operators $\widetilde U_1$ and $ \widetilde U_{-1}$ are represented by with respect to the basis $\{\vec{e}_n\}_{n \geq 0}$ by a matrix with real entries. The last identity says  that the matrix representations in the basis $\{\vec{e}_n\}_{n \geq 0}$ of  $\widetilde U_{i}$ and $\widetilde U_{-i}$ satisfy 
 $$\langle \widetilde U_{-i}\vec{e}_m,\vec{e}_n\rangle = \overline{\langle \widetilde{U}_{i}^{*} \vec{e}_{m}, \vec{e}_{n}\rangle,}$$ which we write as $\widetilde U_{-i}=\widetilde U_{i}^{\#}$.
Therefore, any conjugation 
$\widetilde C$ on $\mathscr{L}^2(\mu, \ell^2)$ such that $\widetilde C(V\mathcal{F}V^{*})=(V\mathcal{F}V^{*})\widetilde C$
can be represented as
$$\widetilde C=
  \begin{bmatrix}
    \widetilde U^\R_{1}\widetilde J & 0 & 0 & 0\\
    0 & 0 & 0 & \widetilde U_{i}^{\#}\widetilde J \\
    0 & 0 & \widetilde U^\R_{-1}\widetilde J &  0\\
    0 & \widetilde U_i\widetilde J & 0 & 0 \\
  \end{bmatrix},
$$
where $U^\R_1$ and $U^\R_{-1}$ are arbitrary unitary operators on $\ell^2$ whose matrix representations with respect to $\{\vec{e}_n\}_{n \geq 0}$ have real entries and $\widetilde U_i$ is arbitrary.
Finally, a conjugation $ C$ on $L^2(\R)$ fulfils the condition  $C\mathcal{F}=\mathcal{F} C$ if and only if
it  is represented with respect  to the decomposition in  \eqref{e_fur} as
\[ C=
\left[
  \begin{array}{cccc}
    U^\R_{1} J & 0 & 0 & 0\\
    0 & 0 & 0 & U_{i}^{\#} J \\
    0 & 0 & U^\R_{-1} J & 0 \\
    0 & U_i J & 0 & 0 \\
  \end{array}
\right]=
\left[
\begin{array}{cccc}
    U^\R_{1}  & 0 & 0 & 0\\
    0 & 0 & 0 & U_{i}^{\#} \\
    0 & 0 & U^\R_{-1}  &  0\\
    0 & U_i  & 0 & 0 \\
  \end{array}
\right]\, J,
\]
where $U^\R_1, U^\R_{-1}$ are arbitrary unitary operators on their respective eigenspaces $\ker (\mathcal{F} - I)$ and $\ker (\mathcal{F} + I)$ which are  represented in terms of the  basis $\{H_{4n}\}_{n \geq 0}$ and   $\{H_{4n+2}\}_{n \geq 0}$ by real matrices, $U_i$ is an arbitrary unitary operator on $\ker (\mathcal{F} + iI)$ and $U_{i}^{\#}$ is the unitary operator on $\ker (\mathcal{F} - i I)$ defined by
$$\langle U_{i}^{\#}H_{4m+3},H_{4n+3}\rangle =\overline{\langle  U_{i}^{*}H_{4m+1}, H_{4 n+1}\rangle}, \quad m, n \geq 0.$$
\end{Example}

\begin{Example}\label{Hilertrt2}
Suppose $\mathscr{H}$ is the Hilbert transform on $L^2(\R)$. Since $\sigma(\mathscr{H})  = \{i, -i\}$ then
$L^2(\R) = \mathscr{E}_{i} \bigoplus \mathscr{E}_{-i}$ and
 $\mathscr{E}_{i}$ has orthonormal basis $\mathscr{B}_{i} = \{f_n\}_{n \geq 1}$
$$f_n(x) = \frac{1}{\sqrt{\pi}} \frac{(x + i)^{n - 1}}{(x - i)^{n}}$$
and $\mathscr{E}_{-i}$ has orthonormal basis $\mathscr{B}_{-i} = \{g_n\}_{n \geq 1}$
$$g_{n}(x) = \frac{1}{\sqrt{\pi}} \frac{(x - i)^{n-1}}{(x + i)^{n}}.$$
See \cite[Ch. 12]{MR4545809} for details. In this example we will describe $\mathscr{C}_c(\mathscr{H})$. Note first that $\mathscr{C}_c(\mathscr{H})\not=\varnothing$ (recall Example \ref{JHKSJDF99} and thus $\mathscr{H}\cong \mathscr{H}^*$).

Similarly as in Example \ref{Foureoirtert1} we can identify $L^2(\R)$ with $\mathscr{L}^2(\mu, \ell^2)$, where $\mu=\delta_i+\delta_{-i}$.
Then, the conjugation $\J^\#$ given by equality \eqref{e1.7} is an antilinear extension of operator  $\J^\#f_n=g_n,\ \J^\#g_n=f_n,\  n\geq 1$.
Putting this in matrix form $$\J^\# =
\begin{bmatrix}
 0&J  \\
 J&0
\end{bmatrix},$$ where $J$ is a conjugation on $L^2(\R)$ which fixes all elements of  $\mathscr{B}_{i}$ and $\mathscr{B}_{-i}$.

Moreover, by Theorem \ref{th1.2a}, any conjugation $C$ on $L^2(\R)$  with $C \mathscr{H} C = \mathscr{H}^{*}$ must take the  (block) form
$$C =
\begin{bmatrix}
U_i & 0\\
0& U_{i}^\#
\end{bmatrix}
\begin{bmatrix}
0 & J\\
J& 0
\end{bmatrix}=\begin{bmatrix}
0 & U_iJ\\
U_i^\#J& 0
\end{bmatrix},$$
where $U_i$ is arbitrary unitary operator on $\mathscr{E}_i$  and  $U_i^\#\in \mathcal{B}(\mathscr{E}_{-i})$ defined as $\langle U_i^\#g_m,g_n\rangle =\overline{\langle U_i^*f_m,f_n\rangle} $ similar to the previous example.
\end{Example}

\section{A remark about invariant subspaces}

The first paper in this series \cite{MPRCOUI} classified, for a fixed unitary operator $U$ on $\h$, the subspaces $\mathcal{M}$ of $\h$ for which $C \mathcal{M} \subset \mathcal{M}$ for every $\mathscr{C}_{s}(U)$ (the symmetric conjugations for $U$). These turned out to be the hyperinvariant subspaces for $U$. What are the subspaces $\mathcal{M}$ for which $C \mathcal{M} \subset \mathcal{M}$ for every $C \in \mathscr{C}_{c}(U)$ (the commuting conjugations for $U$)? We have some partial results in this paper (see for example Proposition \ref{98auifdovjlscfergeAA} and Corollary \ref{uuUUSSSS}). However, we do not have a complete and concise characterization. These subspaces seem complicated to describe in the general abstract situation. However, we do have a characterization in the special case where $U = \M_{\xi}$ on $\mathscr{L}^{2}(\mu, \h)$. Recall the notation from Proposition \ref{p3.2a}.

\begin{Theorem}\label{777yyHyyhh6666T}
Suppose $\mu \in M_{+}(\T)$ such that $\mu^{c} \ll \mu$ and $\h$ is a Hilbert space. For a subspace $\mathcal{K}$ of $\mathscr{L}^2(\mu, \h)$ the following are equivalent. 
\begin{enumerate}
\item $\mathfrak{C} \mathcal{K} \subset \mathcal{K}$ for every $\mathfrak{C} \in \mathscr{C}_{c}(\M_{\xi})$;
\item For a fixed conjugation $J$ on $\h$ and $\J^{\#}$ defined as in \eqref{e1.7}, $\mathcal{K}$ is invariant for $\J^{\#}$ and every $\M_{\mathbf F}$, where $\mathbf{F}$ belongs to 
$$\mathscr{L}^{\infty}_{c}(\mu, \mathcal{B}(\h)) := \{ \mathbf{F} \in \mathscr{L}^{\infty}(\mu, \mathcal{B}(\h)): J \mathbf{F}(\xi) J = \mathbf{F}(\xi)^{\#} \mbox{ for $\mu$-a.e. $\xi \in \T$}\}.$$
\end{enumerate}
\end{Theorem}

The proof of this theorem requires a decomposition theorem from \cite[proof of Corollary 3.19]{MR2003221}. We include a proof for completeness and since the form of the decomposition is important for the proof of Theorem \ref{777yyHyyhh6666T}.

\begin{Lemma}\label{fourunit}  Any $A \in \mathcal{B}(\h)$ can be expressed as a positive constant times the sum of four unitary operators on $\h$. 
\end{Lemma}

\begin{proof}
Define 
$$H=\frac{1}{2 \|A\|}(A+A^*) \; \; \mbox{and} \;  \; K= \frac{1}{2 i \|A\|}(A-A^*)$$ and notice that $H$ and $K$ are selfadjoint contractions and thus $I - H$ and $I - K$ are positive and hence have unique positive square roots. 
Thus, 
$$
U_{1,2}=H\pm i(I-H^2)^{\frac 12} \; \; \mbox{and} \; \;  U_{3,4}=i K\pm(I-K^2)^{\frac 12}
$$
are four unitary operators which satisfy
\begin{align*}
A & =\frac{\|A\|}{2}(U_1+U_2+U_3+U_4). \qedhere
\end{align*}
\end{proof}

\begin{proof}[Proof of Theorem \ref{777yyHyyhh6666T}]
The proof of (b) $\Longrightarrow$ (a) follows from a special case of Theorem \ref{th1.2a}. For the proof of (a) $\Longrightarrow$ (b), we begin with the fact that since $\J^{\#} \in \mathscr{C}_{c}(\M_{\xi})$ then $\mathcal{K}$ is invariant for $\J^{\#}$. Moreover, by Theorem \ref{th1.2a} any $\mathfrak{C} \in \mathscr{C}_{c}(\M_{\xi})$ can be written as $\mathfrak{C} = \M_{\U} \J^{\#}$ for some $\U \in \mathscr{L}^{\infty}(\mu, \mathcal{B}(\h))$ that is unitary valued $\mu$-almost everywhere and also satisfies $J \U(\xi) J = \U(\xi)^{\#}$ for $\mu$-almost every $\xi$. Thus, $\mathcal{K}$ is invariant for $\M_{\U} \J^{\#}$ and thus $\M_{\U}$. Now apply Lemma  \ref{fourunit} to any $\mathbf{F} \in \mathscr{L}^{\infty}_{c}(\mu, \mathcal{B}(\h))$ and conclude that $\mathcal{K}$ is invariant for every $\M_{\mathbf{F}}$ where $\mathbf{F} \in \mathscr{L}^{\infty}_{c}(\mu, \mathcal{B}(\h))$.
\end{proof}

\begin{Remark}
In the scalar case $\h = \C$, note that 
$$\mathscr{L}^{\infty}(\mu, \mathcal{B}(\C)) = \{v \in L^{\infty}(\mu): v(\xi) = v(\overline{\xi}) \mbox{ for $\mu$-a.e. $\xi \in \T$}\}.$$
\end{Remark}

\begin{Example}
For {\em the} bilateral shift $M_{\xi}$ on $L^2 = L^2(m, \T)$, we know that every $C \in \mathscr{C}_{c}(M_{\xi})$ takes the form $M_{u} J$, where $(J f)(\xi) = f(\xi)$ and $u \in L^{\infty}$ such that $u(\xi)$ is unimodular and  $u(\xi) = u(\overline{\xi})$ almost everywhere. One can check that examples of subspaces that are invariant for every $C \in \mathscr{C}_{c}(M_{\xi})$ include 
\begin{enumerate}
\item $\{g \in L^2: g(e^{i t}) = 0, |t| \geq \tfrac{\pi}{2}, g(e^{i t}) = g(e^{-it}), |t| < \tfrac{\pi}{2}\};$
\item $\{g \in L^2: g(e^{i t}) = 0, |t| \geq \tfrac{\pi}{2}, g(e^{i t}) = -g(e^{-it}), |t| < \tfrac{\pi}{2}\};$
\item $\{g \in L^2: g(e^{i t}) = 0, |t| \geq \tfrac{\pi}{2}, g(e^{i t}) = g(e^{-it}), \tfrac{\pi}{4} < |t| < \tfrac{\pi}{2}, g(e^{it}) = -g(e^{-it}), |t| < \tfrac{\pi}{4}\}.$
\end{enumerate}
The variety of these spaces convinces us that a concise description of the $C$-invariant subspaces for every $C \in \mathscr{C}_{c}(M_{\xi})$ seems difficult. 
\end{Example}

\bibliographystyle{plain}

\bibliography{references}

\end{document}